\documentclass[12pt]{amsart}

\usepackage{graphicx}
\usepackage{amssymb}
\usepackage{amsthm}
\usepackage{listings}
\usepackage{lineno}
\usepackage{bbold}
\usepackage[margin=3cm]{geometry}

\usepackage{amsmath}%To use \text 
\usepackage[utf8]{inputenc}
\usepackage{hyperref}
\usepackage[capitalize]{cleveref}
\crefname{thm}{Theorem}{Theorems}
\usepackage[export]{adjustbox}
\usepackage{todonotes}
\usepackage{bm}
\usepackage{wrapfig}
\usepackage{float}
\usepackage{mathtools}
\usepackage{aliascnt}
\newaliascnt{eqfloat}{equation}
\newfloat{eqfloat}{h}{eqflts}
\floatname{eqfloat}{Equation}

\newcommand*{\ORGeqfloat}{}
\let\ORGeqfloat\eqfloat
\def\eqfloat{%
  \let\ORIGINALcaption\caption
  \def\caption{%
    \addtocounter{equation}{-1}%
    \ORIGINALcaption
  }%
  \ORGeqfloat
}

\theoremstyle{definition}
\newtheorem{thm}{Theorem}[section]
\newtheorem{prop}[thm]{Proposition}
\newtheorem{lm}[thm]{Lemma}
\newtheorem{cor}[thm]{Corollary}
\newtheorem{obs}[thm]{Observation}
\newtheorem{defin}[thm]{Definition}
\newtheorem{smpl}[thm]{Example}

\newtheorem{prob}[thm]{Problem}
\newtheorem{conj}[thm]{Conjecture}
\newtheorem{rem}[thm]{Remark}
\crefname{lm}{Lemma}{Lemmas}
\crefname{thm}{Theorem}{Theorems}
\crefname{prop}{Proposition}{Propositions}
\crefname{defin}{Definition}{Definitions}
\crefname{rem}{Remark}{Remarks}

\newcommand{\opi}{\vec{\boldsymbol{\pi}}}
\newcommand{\otau}{\vec{\boldsymbol{\tau}}}

\DeclareMathOperator{\Alg}{\mathrm{Alg}}

\DeclareMathOperator{\rest}{\mathbf{res}}

\DeclareMathOperator{\pat}{\mathbf{p}}

\DeclareMathOperator{\spn}{\mathrm{span}}

\usepackage{amsaddr}

\graphicspath{ {../imgs/} }

\begin{document}

\title{Pattern Hopf algebras} % Subtitle

%\author{Raul Penaguiao\footnote{\href{mailto:raul.penaguiao@math.uzh.ch}{raul.penaguiao@math.uzh.ch}}\footnote{Institute of Mathematics, University of Zurich, Winterthurerstrasse 190, Zurich, CH - 8057.}\footnote{{\bf Keywords:} marked permutations, presheaves, species, Hopf algebras, free algebras}\footnote{2010 AMS Mathematics Subject Classification 2010: 05E05, 16T05, 18D10}}

\author{Raul Penaguiao}
\email{raul.penaguiao@math.uzh.ch}
\address{University of Zurich, Institute of Mathematics}
\date{\today} % Date

\begin{abstract}
In this paper, we expand on the notion of \textit{combinatorial presheaf}, first introduced explicitly by Aguiar and Mahajan in 2010 but already present in the literature in some other points of view.
We do this by adapting the algebraic framework of species to the study of substructures in combinatorics.
Afterwards, we consider functions that count the number of patterns of objects and endow the linear span of these functions with a product and a coproduct.
In this way, any well behaved family of combinatorial objects that admits a notion of substructure generates a Hopf algebra, and this association is functorial.
For example, the Hopf algebra on permutations studied by Vargas in 2014 and the Hopf algebra on symmetric functions are particular cases of this construction.

A specific family of pattern Hopf algebras of interest are the ones arising from \textit{commutative combinatorial presheaves}.
This includes the presheaves on graphs, posets and generalized permutahedra.
Here, we show that all the pattern Hopf algebras corresponding to commutative presheaves are free.

We also study a remarkable non-commutative presheaf structure on marked permutations, \textit{i.e.} permutations with a marked element.
These objects have a natural product called inflation, which is an operation motivated by factorization theorems for permutations.
In this paper we find new factorization theorems for marked permutations.
We use these theorems to show that the pattern Hopf algebra for marked permutations is also free, using Lyndon words techniques.
\end{abstract}

\maketitle

\tableofcontents

\section{Introduction}
%
%\todo[inline]{Notation decisions so far}
%
%
%Irreducible elements exist in any associative presheaf.
%Indecomposable elements are permutations and pertain $\oplus$.
%
%On associative presheaves we use $\ast$ for the product $\ast h\odot h \to h $ and $\cdot $ for the product in $\mathcal{G}(h)$.
%
%Permutations will always be represented by Greek letters.
%

The notion of substructure is important in mathematics, and particularly in combinatorics.
In graph theory, minors and induced subgraphs are the main examples of studied substructures.
Substructures of other objects have been studied: set partitions, trees, paths and, to a larger extent, permutations, where the study of patterns leads to the concept of permutation classes.

A priori unrelated, it has been shown that Hopf algebras are a natural tool in algebraic combinatorics to study discrete objects, like graphs, set compositions and permutations.
For instance, the celebrated Hopf algebra on permutations named after Malvenuto and Reutenauer sheds some light on the structure of shuffles in permutations.
Other examples of Hopf algebras in combinatorics that are relevant to this work are the Hopf algebra on symmetric functions (described for instance in \cite{stanley86}), and the permutation pattern Hopf algebra introduced by Vargas in \cite{vargas14}.

%Our first goal here is to construct a Hopf algebra related to the study of patterns in combinatorial objects.
With that in mind, we build upon the notion of species with restrictions, as presented in \cite{aguiar10} by Aguiar and Mahajan, in order to connect these two areas of algebraic combinatorics.
Species with restrictions, or \textit{combinatorial presheaf}, arise by enriching a species with restriction maps.
With this combinatorial data, we produce a functorial construction of a pattern algebra $\mathcal A (h)$ from any given combinatorial presheaf $h$.
By further considering an associative product on our objects, we can endow $\mathcal{A}(h)$ with a coproduct that makes it a bialgebra, and under specific circumstances a Hopf algebra.
The main examples of combinatorial presheaves are words, graphs and permutations.
Examples of associative products on combinatorial objects are the disjoint union on graphs or the direct sum on permutations.

The algebras obtained from a combinatorial presheaf are always commutative.
In analyzing Hopf algebras, it is of particular interest to show that such algebras are free commutative (henceforth, we simply say \textit{free}), and to construct free generators of the algebra structure.
The fact that a Hopf algebra is a free algebra has several applications.
For instance, in \cite{foissy12}, it was shown that any graded free and cofree Hopf algebra is self dual.
Moreover, self dual Hopf algebras are characterized by studying their primitive elements.
The freeness of a Hopf algebra also allows us to gain some insight into the character group of the Hopf algebra; see for instance \cite{supina19}.
It can also be used under duality maps to establish cofreeness of Hopf monoids, as described in the methods of M\"obius inversion in \cite{sanchez19}.

We will use tools from the combinatorics of words and Lyndon words.
In fact, Lyndon words are commonly used to establish the freeness of algebras.
Examples are the shuffle algebra in \cite{radford79} (see also \cite[Chapter 6]{grinberg14}), the algebra of quasisymmetric functions in \cite[Theorem 8.1]{hazewinkel01}, and the algebra on word quasisymmetric functions in non-commutative variables, in \cite{bergeron09}.

Computing the primitive space of a specific Hopf algebra is of interest.
If the Hopf algebra $H = \bigcup_{n\geq 0} H^{(n)}$ is a filtered Hopf algebra with $P(H)$ primitive space, computing the dimension of $H^{(n)} \cap P(H)$ is a classic problem, having applications for instance in establishing that a given Hopf algebra is not a Hopf subalgebra of another one.
Thus, when we establish that a Hopf algebra is in fact a pattern Hopf algebra, we can use these tools to understand the primitive space and the coradical filtration of a Hopf algebra.
Conversely, this also gives us a very flexible tool to establish that a given Hopf algebra is in fact a patter Hopf algebra.
We use precisely this method to establish that the symmetric function Hopf algebra is a pattern Hopf algebra, and we conjecture that the same method gives us that the quasi-symmetric function Hopf algebra is also a pattern Hopf algebra.

In this paper, we show that any commutative combinatorial presheaf gives rise to a pattern algebra that is free commutative.
We also study a non-commutative combinatorial presheaf on marked permutations, where we establish the freeness, construct the free elements with the help of Lyndon words, and enumerate the primitive elements of the pattern Hopf algebra on marked permutations.
In the rest of this section, we present these results in more detail, and describe the methods for proving them.

\subsection{Pattern Hopf algebras from monoids in presheaves}

Let $\mathtt{Set}_{\hookrightarrow}$ be the category whose objects are finite sets and morphisms are injective maps between finite sets.
Let also $\mathtt{Set}_{\times}$ be the category whose objects are finite sets and morphisms are bijective maps between finite sets.
Write $\mathtt{Set}$ for the usual category of finite sets.

A \textbf{species} is a functor from $\mathtt{Set}_{\times}$ to $\mathtt{Set}$ (or equivalently, to itself).
Hence, a species $h$ is described by an assignment of each set $I$ to a finite set $h[I]$, together with some relabeling map for each bijection.
Species occur very naturally in combinatorics as a way of describing the combinatorial structures on finite sets, for instance graph structures on a vertex set or a poset on a ground set.

\begin{defin}[Combinatorial presheaves]
A \textit{combinatorial presheaf} (or a \textit{presheaf}, for short) is a contravariant functor from $\mathtt{Set}_{\hookrightarrow}$ to $\mathtt{Set}$.
A morphism of combinatorial presheaves is simply a natural transformation of functors.
In this way, we have the category $\mathtt{CPSh}$ of combinatorial presheaves.
This was introduced in \cite[Section 8.7.8.]{aguiar10}.
\end{defin}

In this form, a presheaf is simply a species enriched with restriction maps $\rest_J : h[I] \to h[J] $ for each inclusion $J\hookrightarrow I$ is a way that is functorial, that is if $J_1\subseteq J_2$, then $\rest_{J_2} \circ \rest_{J_1} = \rest_{J_1}$.
The notion of \textbf{presheaves} has been around in category theory and geometry for some time, where it generally refers to contravariant functors from the category of open sets of a topology with inclusions as morphisms.
The main examples of combinatorial presheaves are graphs, set compositions, and permutations; see \cref{smpl:permpresheaf}.
In general, any combinatorial object that admits a notion of restriction admits a presheaf structure.

The category $\mathtt{CPSh}$ and monoidal and bimonoidal structures thereof have been studied in \cite{aguiar10}. 
We will explain the connections between this prior work and our results in \cref{rem:aguiar}.

In presheaves, two objects $a\in h[I], b\in h[J]$ are said to be \textit{isomorphic objects}, or $a\sim b$, if there is a bijection $f:I\to J$ such that $h[f](b)=a$.
The set of equivalence classes in $h[I]$ is denoted $h[I]_{\sim}$.
Let $h[n]$ denote the objects of type $h$ on the set $[n] = \{1, \dots , n\}$.
The collection of equivalence classes of a presheaf $h$ is denoted by $\mathcal{G}(h) = \bigcup_{n\geq 0 } h[n]_{\sim }$.
In this way, the set $\mathcal G(h) $ is the collection of all the $h$-objects up to isomorphism.

If $b \in h[I]$ and $J \subseteq I$, we denote $b|_J$ for $h[\mathrm{inc}](b)$, where $\mathrm{inc}: J \hookrightarrow I$ is the natural inclusion map.

\begin{defin}[Patterns in presheaves]\label{defin:pattern}
Let $h$ be a presheaf, and consider two objects $a\in h[I], b\in h[J]$.
We say that $J'\subseteq J $ is a \textit{pattern} of $a$ in $b$ if $b|_{J'} \sim a$.
We define the \textit{pattern function} $$\pat_a( b) : = \left| \{J' \subseteq J \text{ such that } \rest_{J'}(b) \sim a \} \right| \,  . $$

\end{defin}

Fix a field $k$ with characteristic zero.
Denote the family of functions $A \to B$ by $\mathcal{F} (A, B)$.

This definition only depends on the isomorphism classes of $a$ and $b$; see \cref{prop:welldef}.
Hence, we can consider $\{ \pat_a \}_{a\in \mathcal{G}(h)}$ as a family of functions from $\mathcal{G}(h)$ to $k$, indexed by $\mathcal{G}(h)$.
In \cref{smpl:permpresheaf}, we see an example of a presheaf structure on permutations.

In the following we will be using the field of rational numbers, but to the goals of this paper it is only relevant to pick a field $\mathbb{k}$ that has characteristic zero.

If $h$ is a combinatorial presheaf, then the linear span of the pattern functions  is a linear subspace $\mathcal{A}(h) \subseteq \mathcal{F}(\mathcal{G}(h) , \mathbb{Q}) $ of the space of functions $\mathcal{G}(h) \to \mathbb{Q}$.

\begin{thm}\label{thm:algfunctor}
The vector space $\mathcal{A}(h)$ is closed under pointwise multiplication and has a unit.
It forms an algebra, called the \textit{pattern algebra}.
More precisely, if $a, b \in \mathcal G(h)$,
\begin{equation}\label{eq:prodrule}
\pat_ a   \pat_b = \sum_c \binom{c}{a, b} \pat_c \, ,
\end{equation}
where the coefficients $\binom{c}{a, b}$ are the number of ``quasi-shuffles'' of $a, b$ that result in $c$, specifically, if we take $c\in h[C]$ to be a representative of the equivalence class $c$, then:
$$ \binom{c}{a, b} = \left| \{(I, J) \, \text{ such that } \, \,  I \cup J = C \, ,\, \, c|_{I} \sim a, \, c|_{J} \sim b \} \right| \, .  $$
\end{thm} 

Quasi-shuffles of objects have been studied in several contexts as a notion of merging objects together.
For details on quasi-shuffles of combinatorial objects, the interested reader can see \cite{hoffman00,aguiar10,foissy16}.

\bigskip

We introduce now the \textit{Cauchy product} $\odot $ on the category of combinatorial presheaves, which associates to two combinatorial presheaves $h$ and $ j$ the combinatorial presheaf 
$$h \odot j : I \mapsto \biguplus_{A\uplus B = I} h[A] \times j[B] \, . $$
The unit for this product is the unique presheaf that satisfies $\mathcal{E}[A] = \emptyset$ for $A \neq \emptyset$, and $\mathcal{E}[\emptyset] = \{ \diamond \} $.
If $f:h_1 \Rightarrow h_2$ and $ g: j_1 \Rightarrow j_2 $ are natural transformations of presheaves, then $f\odot g$ is a natural transformation such that $(f\odot g)_I$ is the natural mapping from $h_1[A]\times j_1[B]$ to $h_2[A] \times j_2[B]$, for each decomposition $A\uplus B = I$.

Product structures on categories were examined in \cite{maclane63}.
There, it was shown that the Cauchy product gives the category of presheaves a monoidal structure. We call an object in this monoidal category an \emph{associative presheaf}.

This is a triple $(h, \ast, 1)$, where $h$ is a combinatorial presheaf, $\ast $ is a natural transformation $h\odot h \Rightarrow h$, and $1 \in h[\emptyset ] $ a unit that satisfy classical axioms of associativity and unit. 
We explain the details further in \cref{sec:prel} below.

Examples of associative operations on combinatorial presheaves are the disjoint union of graphs and the direct sum of permutations.
Another less standard example, which we study in this paper, is the inflation of marked permutations, defined below.

Observe that the associative product $\ast $ in our combinatorial objects is a natural transformation.
This means that the product is stable with respect to relabelings and restrictions, so we can also define the corresponding product on $\mathcal{G}(h)$, which we denote by $\cdot $ for the sake of distinction; see \cref{defin:prodonG}.
With this, we introduce the following coproduct in the pattern algebra $\mathcal A (h)$:
\begin{equation}\label{eq:coprodformula}
\Delta \pat_ a = \sum_{\substack{ b, c\in \mathcal G (h) \\ a = b \cdot c}} \pat_b \otimes \pat_c \, .
\end{equation}
where the sum runs over coinvariants $b, c$ such that $a = b \cdot c$.
The main property that motivates this operation in $\mathcal{A}(h)$ is that, under the natural identification of the function algebra $\mathcal{F}(\mathcal G (h), \mathbb Q)^{\otimes 2} $ as a subspace of $\mathcal{F}(\mathcal G (h) \times \mathcal G (h), \mathbb Q) $, we have
\begin{equation}\label{eq:coproddefin}
 \Delta \pat_a (b,  c) =  \pat_a (b \cdot c) \, .
\end{equation}
This is shown in \cref{thm:conHopfalgebra}.
The relation \eqref{eq:coproddefin} is central in establishing that the coproduct $\Delta $ is compatible with the product in $\mathcal A (h)$.

\begin{thm}\label{thm:conHopfalgebra}
Let $(h, \ast, 1) $ be an associative presheaf.
Then the pattern algebra of $h$ together with this coproduct, and a natural choice of counit, forms a bialgebra.
If additionally $| h[\emptyset ] | = 1 $, the pattern algebra forms a Hopf algebra.
\end{thm}

Presheaves that satisfy $|h[\emptyset ]| = 1$ are called \textit{connected}.
Connected algebraic structures are a classical resource in graded Hopf algebras, as in this way we can find an antipode through the so called \textit{Takeuchi formula}, introduced in \cite{takeuchi71}.

Some known Hopf algebras can be constructed as the pattern algebra of a combinatorial presheaf.
An example is $Sym$, the Hopf algebra of \textit{symmetric functions}.
This Hopf algebra has a basis indexed by partitions, and corresponds to the pattern Hopf algebra of the presheaf on set partitions (see details in \cref{sec:spartpatalg}).
The pattern Hopf algebra corresponding to the presheaf on permutations described above was introduced by Vargas in \cite{vargas14}.
Some other Hopf algebras constructed here, like the ones on graphs and on marked permutations below, are new, and some we conjecture are isomorphic to known Hopf algebras like the pattern Hopf algebra on set compositions, which may be simply the Hopf algebra of quasi-symmetric functions; see \cref{conj:QSym} below.

We also establish some general properties of pattern Hopf algebras, like describing their primitive elements, finding the inverse of the so-called \textit{pattern action} in $\mathcal{A}(h)$, and relating $\mathcal A (h)$ with the Sweedler dual of an algebra generated by $\ast $.

\begin{rem}\label{rem:aguiar}
In the terminology of \cite{aguiar10}, a cocommutative comonoid in set species is precisely what we call here a combinatorial presheaf.
This is done by identifying $\Delta_{A, B}(a) = a|_A \times a|_B$ with the restrictions of $a$ to the sets $A$ and $B$.
Furthermore, an associative presheaf is a cocomutative bimonoid in set species.
The coalgebra structure of the pattern Hopf algebras that we construct here is a subcoalgebra of the dual algebra of the so called \textit{bosonic Fock functor} of these comonoids in linearized set species.
This is the subject of \cite[Proposition 8.29]{aguiar10}.
However, the algebra structure is in general different.

Specifically, on the combinatorial presheaf on graphs introduced below, the corresponding coalgebra structure is the dual of the well known incidence Hopf algebra introduced in \cite{schmitt94}.
\end{rem}

\subsection{Commutative presheaves\label{sec:introcomuandstrat}}

In this paper, we focus on the problem of proving the freeness of some pattern algebras.

\begin{defin}[Free generators of an algebra]
If $A$ is a commutative algebra, a set $\mathcal G \subset A$ is a generating set if the smallest subalgebra $A$ containing $\mathcal G$, written $\langle \mathcal G \rangle$, is $A$.

A set $\mathcal G \subset A$ is a (commutative) free set if for any finite collection $y_1, \dots , y_k$ of elements in $\mathcal G$ and any non-zero polynomial in $k$ commutative variables $p = p(x_1, \dots , x_k)$, the evaluation $p(y_1, \dots , y_k)$ is a non-zero element in $A$.

If an algebra $A$ has a free generator set $\mathcal G$, the algebra is \textit{free}.
\end{defin}

The first case that we want to explore is the one of commutative presheaves.
An associative presheaf $(h, \ast, 1)$ is called \textit{commutative} if $\ast $ is commutative, that is for any $a\in h[I], b\in h[J]$ we have that $a \ast b = b \ast a$; see \cref{def:asspresheaf}.

As it turns out, this is enough to guarantee the freeness of the pattern Hopf algebra. 
This is one of the main results of this paper.

\begin{thm}[Commutative combinatorial presheaves are free]\label{thm:comutative}
The pattern Hopf algebra of an associative combinatorial presheaf with a commutative product $\ast $ is free.
The free generators are the pattern functions indexed by the irreducible objects with respect to $\ast $.
\end{thm}

The proof of this result is presented in \cref{sec:comutfree}.
The main ingredient for this result is \cref{cor:factsthm}, a surprising structure result on associative presheaves, which shows that the structure of an associative presheaf can be described much in the same way as we can describe a group: by providing a set of generators and a set of relations that these generators must satisfy.

In the case of the graph pattern Hopf algebra, \cref{thm:comutative} was already proved in \cite[Theorem 3]{whitney1932coloring}, where a function on graphs that satisfies \eqref{eq:prodrule} is shown to be completely determined by its values on the connected graphs, that the remaining values are obtained via polynomial expressions, and that no other polynomial expressions hold for such a generic function.
%There in fac

\subsection{Non-commutative presheaves}

As mentioned above, the pattern Hopf algebra on permutations is the one discussed by Vargas in \cite{vargas14}, where it is shown that it is free.
This is an example of a non-commutative associative presheaf.
There, free generators were constructed.
These generators correspond to Lyndon words of $\oplus$-indecomposable permutations, see \cite{chen58} for an introduction to combinatorics of Lyndon words.

%The free generators are constructed as follows: let $\mathcal I $ be the family of $\oplus $ indecomposable permutations, endowed with a total order.
%To a word in $\mathcal I $ it corresponds a permutation by applying the $\oplus $ operation.
%The free generators are those permutations corresponding to Lyndon words on the alphabet $\mathcal I $.

In this paper we explore other associative presheaves that are non-commutative.
Taking the presheaf on permutations as our starting point, we wish to study monoidal structures that are more complex than the $\oplus $ product (see \cref{smpl:assperm}), but still allows us to establish the freeness property.
We investigate the presheaf on marked permutations $\mathtt{MPat}$, which is equipped with the inflation product.
This product is motivated by the inflation procedure on permutations described in \cite{albert03}.
Its presheaf structure is presented below in \cref{smpl:mpermutations}.
This is not a commutative presheaf, so \cref{thm:comutative} does not apply.
However, we still have the following: 

\begin{thm}[Freeness of $\mathcal{A}(\mathtt{MPat})$]\label{thm:Hopffunctor}
The pattern algebra $\mathcal{A}(\mathtt{MPat})$ is free.
\end{thm}

%
%The investigation of marked permutations is motivated by the fact that labeled permutations and decorated permutations are used throughout combinatorics.
%Furthermore, this seems to be a much harder Hopf algebra to establish freeness than, say, the permutation pattern Hopf algebra, because the factorization theorem that the method is based upon was not yet present in the literature.
%In \cref{sec:strat}, we present a blueprint for a proof of freeness.

To establish the freeness of $\mathcal{A}(\mathtt{MPat})$ we present a unique factorization theorem on marked permutations with the inflation product in \cref{cor:simpleUFT}.
This is an analogue of \cite[Theorem 1]{albert03}.% for the inflation, but on the context of marked permutations.
With it, we find generators of the algebra on marked permutations, and use the Lyndon factorization of words in \cite{chen58} to show that these generators are free generators.
%A \textit{Lyndon word} was introduced in \cite{chen58} as a word $w$ that is lexicographically smaller than all its suffixes.
%
%In \cref{sec:mperfree} we also enumerate the dimension of the primitive space of the pattern Hopf algebra $\mathcal{A}(\mathtt{MPer})$

In this way, the presheaf on marked permutations will be one main focus of this paper: \cref{sec:mperfree} will be dedicated to the freeness problem on this presheaf, as well as enumerating the dimension of the primitive space of the pattern Hopf algebra, which corresponds to enumerating the marked permutations that are irreducible with respect to the inflation product.

\section{Notation and preliminaries\label{sec:prel}}

\subsection{Species and monoidal functors}

%If $h$ is a combinatorial presheaf, we are given a collection $h[I]$, for each finite set $I$, of elements that we call \textit{objects}.
%These are thought of as the objects of type $h$ based on the set $I$.
%Additionally, for each bijection $\sigma: I \to J$ we are also given $h[\sigma]: h[J] \to h[I]$, the relabeling maps.
%If $| I | = | J | $, and $a\in h[I], b\in h[J]$, we write $a\sim b $ if there is a bijection $\sigma: I \to J $ such that $h[\sigma](b) = a$.
%Note that this defines an equivalence relation on the objects of type $h$.
%In particular, $S_I $ acts on each $h[I]$.
%We define a \textit{coinvariant} as an equivalence class of $\sim $, and write $h[I]_{\sim }$ for the coinvariants in $h[I]$.
%The information defines indeed a species, a functor $h: \mathtt{Set}_{\times} \to \mathtt{Set}$.
%We define $\mathcal{G}(h) = \bigcup_{n\geq 0} h[n]_{\sim }$, the family of coinvariants.
If $h$ is a combinatorial presheaf, and if $a \in h[I]$, we define the \textit{size} of $a$ as $|a| := | I | $ and the \textit{indexing set} of $a$ as $\mathbb{X}(a) = I$.

%The above information describes a \textbf{set species}.
%A presheaf further extends such notion by setting, given $J\subseteq I$ and $a\in h[I]$, a restriction $a|_J := h[J\hookrightarrow I](a)$ in a functorial way.
%That is, if $I_1 \subseteq I_2 \subseteq I$, then $a|_{I_2} |_{I_1} = a|_{I_1}$.
%This allows us to talk about patterns and substructures in the objects of type $h$.
%The restriction maps, together with the species structure, define a functor $h: \mathtt{Set}_{\hookrightarrow{}} \to \mathtt{Set}$.

%
%\begin{defin}[Patterns in presheaves]
%Let $h$ be a combinatorial presheaf, and $a\in h[I]$.
%If $b \sim a|_J $ for some $J\subseteq I$, we call $b $ a \textit{pattern} of $a $ and write $b \leq a$.
%In this case we call $J$ an \textit{occurrence} of $b $ in $a $.
%\end{defin}

\begin{smpl}[The unit presheaf]
The species $\mathcal{E} $ with only one object $\diamond $, which has size zero, has a unique presheaf structure.
% and is the unit for the Cauchy product.
%The presheaf of graphs 
%$$\mathtt{Gr}[I]=\{\text{graphs with vertex set } I \}\, , $$
%results from the usual species structure by adding the natural graph restrictions.
%
%In this way, $\mathbb{X}(\diamond ) = \emptyset $, and if $G$ is a graph in the vertex set $V$, we have that $\mathbb{X}(G) = V $ and $|G| = |V|$.
\end{smpl}

\begin{smpl}[The presheaf on permutations]\label{smpl:permpresheaf}

To fit the framework of presheaves, we use a rather unusual definition of permutations introduced in \cite{albert18}.
There, a permutation on a set $I$ is seen as a pair of orders in $I$.
This relates to the usual notion of a permutation as bijections in the following way: if we order the elements of $I = \{a_1 \leq_P \dots \leq_P a_k \} = \{ b_1 \leq_V \dots \leq_V b_k \}$, then this defines a bijection via $a_i \mapsto b_i $ in $I$.
Conversely, for any bijection $f$ on $I$, there are several pairs of orders $(\leq_P, \leq_V)$ that correspond to the bijection $f$, all of which are isomorphic.
If $I \hookrightarrow J$, by restricting the orders on $I$ to orders on $J$ via the injective map $\hookrightarrow $, we obtain a restriction to a permutation on $J$.
The resulting presheaf structure is denoted by $\mathtt{Per}$.

It will be useful to represent permutations in $I$ as square diagrams labeled by $I$.
This is done in the following way: we place the elements of $I$ in an $|I| \times |I|$ grid so that the elements are placed horizontally according to the $\leq_P$ order, and vertically according to the $\leq_V$ order.
For instance, the permutation $\pi = \{1<_P2<_P3 , 2<_V1<_V3\}$ in $\{1, 2, 3\}$ can be represented as 
\begin{equation}
\begin{array}{|c|c|c|}
	\hline & & 3 \\
    \hline 1 & &  \\
    \hline & 2 & \\
    \hline 
\end{array}\, \, \, .
\end{equation}

In this way, there are $(n!)^2 $ elements in $\mathtt{Per}[n]$.
Up to relabelling, we can represent a permutation as a diagram with one dot in each column and row.
Thus, $\mathcal{G}(\mathtt{Per})$ has $n!$ isomorphism classes of permutations of size $n$, as expected.

\label{defin:per}
If we consider a permutation $\pi$ on a set $I$, that is, a pair $(\leq_P, \leq_V) $ of total orders in $I$, we write $\mathbb{X}(\pi) = I$.
If $f:J \to I $ is an injective map, the preimage of each order $\leq_P, \leq_V$ is well defined and is also a total order in $J$.
This defines the permutation $\mathtt{Per}[f](\pi )$.

A crucial observation is that this notion recovers the usual concept of permutation pattern already present in the literature.
Specifically, the number of occurrences of a permutation $\tau $ in $\pi$ as described in \cite{wilf02} is precisely $\pat_{\tau }(\pi)$, and a permutation $\pi$ \textit{avoids} $\tau$ in the sense describe in \cite{knuth11} if $\pat_{\tau }(\pi)=0$.
\end{smpl}

Recall that we are given a bifunctor $\odot $ that endows the category of combinatorial presheaves with a monoidal category structure, as introduced in \cite{aguiar10}.
%We start by introducing the Cauchy product of two presheaves $a, b$.
%This is the combinatorial presheaf $g\odot h$ defined on the set $I$ as
%$$ g\odot h [I] =\biguplus_{I = A \uplus B} g[A] \times h[B] \, . $$
%If $f_1: g_1 \Rightarrow h_1 , f_2: g_2 \Rightarrow h_2 $ are two morphisms of presheaves, then $f_1 \odot f_2$ is defined as expected, specifically, for $A, B $ disjoint sets and $a\in h[A]$, $b\in h[B]$, we let 
%$$ f_1 \odot f_2 (a, b) = (f_1(a), f_2(b)) \, .$$
%
%This forms the monoidal category of combinatorial presheaves $\mathtt{CPSh}$.
%Monoidal categories are introduced in combinatorics in \cite{aguiar10}.

\begin{defin}[Associative presheaf]\label{def:asspresheaf}
An \textit{associative presheaf} is a monoid in $\mathtt{CPSh}$, that is, is a combinatorial presheaf $h$ together with natural transformations $\eta: h \odot h \Rightarrow  h$ and $ \iota: \mathcal{E} \Rightarrow h$ that satisfy associativity and unit conditions.
We use, for $a\in h[I]$ and $b\in h[J]$, the notation $\eta_{I, J}(a, b) = a\ast b$.
We also denote the unit by $1 := \iota[\emptyset ](\diamond )\in h [ \emptyset ] $.

This is \textit{commutative} if, for any $a\in h[I], b\in h[J]$ with $I\cap J =\emptyset $, we have $a\ast b = b \ast a$.
\end{defin}

Thus, a product on a presheaf simply describes how to \textit{merge} objects of a certain type $h$ that are based on disjoint sets.

\begin{obs}[Naturality axioms in associative presheaves]\label{obs:naturality}
The naturality of $\eta $, the associativity and unit conditions correspond to the following properties:
\begin{itemize}

\item For all $I, J$ disjoint sets, all $a\in h[I], b\in h[J]$ and all $A\subseteq I, B\subseteq J$, we have
$(a\ast b)|_{A\sqcup B} = a|_A \ast b|_B$.

\item For all $I, J, K$ disjoint sets and all $a\in h[I], b\in h[J], c\in h[K]$, we have $(a \ast b) \ast c = a \ast (b\ast c)$.

\item For any set $I$, and $a\in h[I]$, we have $a\ast 1 = 1 \ast a = a$.
\end{itemize}
\end{obs}

\begin{rem}[Monoidal product and quasi-shuffle]\label{rem:inflqshuf}
In an associative presheaf $(h, \ast, 1)$, let $a\in h[I], b \in h[J]$ with $I, J$ disjoint sets.
Then it is not always the case that $a|_{\emptyset } = 1$, that $(a\ast b)|_I = a$ or that $
(a\ast  b	) |_J = b$.

This is the case, however, when $h$ is a connected presheaf.
It follows that in an associative connected presheaf $h$, we have that $\binom{a \ast b}{a, b} \geq 1$.
\end{rem}

\begin{defin}
Let $(h, \ast_h, 1_h)$ and $(j, \ast_j, 1_j)$ be associative presheaves, and let $f$ be a \textit{presheaf morphism} between $(h, \ast_h, 1_h)$ and $(j, \ast_j, 1_j)$.
This is an \textit{associative presheaf morphism} if it maps unit to unit, and the associative product of the associative presheaves.
\end{defin}

That is, $f: h \Rightarrow j$ is an associative presheaf morphism if it is a presheaf morphism that satisfies $f(1_h)=1_j$ and $f(b' \ast_h  c') = f(b') \ast_j f( c')$ for any $b'\in h[I], c' \in h[J]$.

\subsection{Preliminaries on permutations, graphs and marked permutations}

\begin{smpl}[Presheaf on graphs]\label{smpl:grphs}
There are many notions of patterns on graphs: minors, subgraphs and induced subgraphs are among some of those.
The one that forms a presheaf is the one of induced subgraphs.

This forms a combinatorial presheaf $\mathtt{Gr}$ that can be endowed with the associative product of disjoint union of graphs $\uplus$.
This is in fact a commutative product, so $\mathcal A ( \mathtt{Gr} )$ is a free algebra, and the pattern functions of connected graphs are the free generators, that is:
$$\mathcal{A}(\mathtt{Gr}) = k[ \pat_G | \, G \text{ connected graph} ] \, .$$
\end{smpl}

\begin{smpl}[Permutations and their pattern Hopf algebra]\label{smpl:assperm}
To the presheaf $\mathtt{Per}$ there corresponds a pattern algebra $\mathcal{A}(\mathtt{Per})$ as discussed above.
We can further consider $\mathtt{Per} $ with a monoid structure via the direct sum of permutations $\oplus $, defined as follows:
Suppose that $\pi\in \mathtt{Per}[I], \tau \in \mathtt{Per}[J]$ are two permutations based on the disjoint sets $I, J$, respectively.
The permutation $\pi \oplus \tau \in \mathtt{Per}[I\sqcup J]$ is the pair of total orders $(\leq_P^{\oplus }, \leq_V^{\oplus }) $ extending both of the respective orders from $\pi, \tau $ to $I\sqcup J$ by forcing that $i\leq_P^{\oplus } j$ and $i\leq_V^{\oplus } j$ for any $i\in I, j\in J$.
Correspondingly, the diagram of $\pi \oplus \tau $ results from the ones from $\pi $, $\tau $ as follows
\begin{equation*}
\pi \oplus \tau = \begin{array}{|c|c|}
	\hline & \tau\\
    \hline \pi &  \\
    \hline
\end{array}\, \, \, .
\end{equation*}

We note that this is not a commutative presheaf: in general, $\pi \oplus \tau $ is a different permutation than $\tau\oplus \pi $.
\end{smpl}

%With this, isomorphic classes of permutations is what we usually call a permutation on $I$.

We can write a permutation in its two-line notation, as $ \substack{a_1, \dots , a_k \\ b_1 , \dots , b_k }$ where $a_1 \leq_V a_2 \leq_V \dots $ and $b_1 \leq_P b_2 \leq_P \dots \leq_P b_k$.
If we identify $b_1 , \dots , b_k $ with $1, \dots , k$, respectively, we can disregard the bottom line.
This also disregards the indexing set $I$, and in fact any two isomorphic permutations have the same representation with the one line notation.

We call the unique permutation in the empty set the \textit{trivial permutation} and denote it $\emptyset$.

\begin{defin}[The $\ominus$ operation]
Given two permutations, $\pi, \sigma$, we have already introduced the product $\pi \oplus \sigma$.
We now define the permutation $\pi \ominus \tau \in \mathtt{Per}[I\sqcup J]$ as the pair of total orders $(\leq_P^{\ominus }, \leq_V^{\ominus }) $ extending the respective ones from $\pi, \tau $ to $I\sqcup J$ by forcing that $i\leq_P^{\ominus } j$ and $i\geq_V^{\ominus } j$ for any $i\in I, j\in J$.
Correspondingly, the diagram of $\pi \ominus \tau $ results from the ones from $\pi $, $\tau $ as
\begin{equation*}
\pi \ominus \tau = \begin{array}{|c|c|}
	\hline \pi & \\
    \hline &  \tau \\
    \hline
\end{array}\, \, \, .
\end{equation*}
\end{defin}

It is a routine observation to check that both $\oplus, \ominus $ are associative products on $\mathtt{Per}$, and that $\emptyset $ is the unit of both operations, by simply checking that all properties in \cref{obs:naturality} are fulfilled.

\begin{smpl}[Marked permutations and their pattern Hopf algebra]\label{smpl:mpermutations}
A marked permutation $\pi^*$ on $I$ is a pair of orders $(\leq_P, \leq_V)$ on the set $I \sqcup \{ * \}$.
Intuitively, this gives us a rearrangement of the elements of $I \sqcup \{ * \}$, where one element is special and marked.
The relabelings and restriction maps are the natural ones borrowed from orders, giving us a combinatorial presheaf, that we call $\mathtt{MPer}$.
We can represent a marked permutation in a diagram, as we do for permutations.
Note that in this case the marked element $*$ never changes position after relabelings.
Take for instance the marked permutations $\pi^* = (1<_P2<_P *, 2<_V1<_V* )$, $\tau^* = (*<_P1<_P2, 1<_V*<_V2)$, and $\sigma^* = (*<_P2<_P1, 2<_V*<_V1)$.
Observe that there is no isomorphism between $\pi^* $ and $\tau^* $, whereas there is one between $\tau^* $ and $\sigma^* $, via the relabeling $1 \mapsto 2, 2 \mapsto 1$.

\begin{equation*}
\pi^* = \begin{array}{|c|c|c|}
	\hline & & * \\
    \hline 1 & & \\
    \hline & 2 & \\
    \hline
\end{array}\, \, \, \, \, \, \, \, \, \, 
\tau^* = \begin{array}{|c|c|c|}
	\hline & & 2 \\
    \hline * & & \\
    \hline & 1 & \\
    \hline
\end{array}\, \, \, \, \, \, \, \, \, \, 
\sigma^* = \begin{array}{|c|c|c|}
	\hline & & 1 \\
    \hline * & & \\
    \hline & 2 & \\
    \hline
\end{array}\, \, 
\, .
\end{equation*}

In this way, there are $((n+1)!)^2 $ elements in $\mathtt{MPer}[n]$.
Up to relabeling, we can represent a marked permutation as a diagram with one dot in each column and row, where a particular dot is the distinguished element $*$.
Therefore, $\mathcal{G}(\mathtt{MPer})$ has $(n+1) \times (n+1)! $ isomorphism classes of marked permutations of size $n$.
\end{smpl}

\begin{defin}[Marked permutations]\label{defin:mper}
If we consider a marked permutation $\pi^*$ on a set $I$, that is, a pair $(\leq_P, \leq_V) $ of total orders in $I^*$, we write $\mathbb{X}(\pi^*) = I$.
If $f:J \to I $ is an injective map, this can be extended canonically to an injective map $f^* : J^* \to I^*$.
Thus, the preimage of each order $\leq_P, \leq_V$ under $f^* $ is well defined and is also a total order in $J^*$.
This defines the marked permutation $\mathtt{MPer}[f](\pi^* )$.

Note that a relabeling of the permutation $\pi^*$ in $I^*$ is a relabeling of the corresponding marked permutation in $I$ if the relabeling preserves the marked elements.
\end{defin}

We can also write marked permutations in a one line notation, where we add a marker over the position of $*$.
The resulting notation only disregards the indexing set $I$, and so any two isomorphic marked permutations have the same one line notation.
Note that for each permutation of size $n$ it corresponds $n$ different non-isomorphic marked permutations of size $n-1$, one for each possible marked position.

\begin{smpl}
If we consider $(1 <_P 2 <_P * <_P 4,\, \, 1 <_V *  <_V 4 <_V 2)$, a marked permutation on $\{ 1, 2, 4\}$, its representation with the one line notation is $14\bar{2}3$.%\todo[inline]{Comma is a bit weird separation numbers}

The marked permutation $\tau^* = (73 <_P x <_P * <_P 47, \, \, 73 <_V * <_V x <_V 47 )$ is based on the set $I= \{x, 47, 73 \}$ and has a one line representation $13\bar{2}4$.
Consider now $\pi^* = (1 <_P * <_P 2, 1 <_V * <_V 2) $ and $\sigma^* = (1 <_P * <_P 2, * <_V 1<_V 2)$ marked permutations in $\{1, 2 \}$
So the marked permutations $\tau^*, \pi^*, \sigma^*$ correspond to the one line notations below
\begin{equation}\label{eq:infsmpl}
\tau^* =  13\bar{2}4 = 
\begin{array}{|c|c|c|c|}
	\hline
    & & & \cdot \\
    \hline & \cdot & & \\
    \hline & & \odot & \\
    \hline \cdot & & & \\
    \hline 
\end{array} \, , \, \,\pi^* = 1\bar{2}3 = 
\begin{array}{|c|c|c|}
	\hline & & \cdot \\
    \hline & \odot & \\
    \hline \cdot & & \\
    \hline 
\end{array}\, ,  \, \, \sigma^* = 2\bar{1}3 = 
\begin{array}{|c|c|c|}
	\hline & & \cdot \\
    \hline \cdot & & \\
    \hline & \odot & \\
    \hline 
\end{array}\, .
\end{equation}
%For instance, in the previous example in \cref{smpl:msetsmperms}, we have that $\w(\tau^* ) = 231$.

Then, we have that $\pi^*, \sigma^*$ are patterns of $ \tau^* $, because $J = \{73, 47 \}$ is an occurrence of $\pi^* $ in $\tau^* $, and $J' = \{x, 47\}$ is an occurrence of $\sigma^* $ in $\tau^*$.
\end{smpl}

\begin{defin}[Inflation product]
The inflation product $\star $ in marked permutations is defined as follows:
Given two marked permutations $\tau^* \in \mathtt{MPer}[I]$ and $\pi^* \in \mathtt{MPer}[J]$ with $I, J$ disjoint sets, the inflation product $\tau^* \star \pi^* \in \mathtt{MPer}[I\sqcup J]$ is a marked permutation resulting from replacing in the diagram of $\tau^* $ the marked element with the diagram of $\pi^* $.
Here is an example:
\begin{equation*}
\includegraphics[valign=c, scale=0.65]{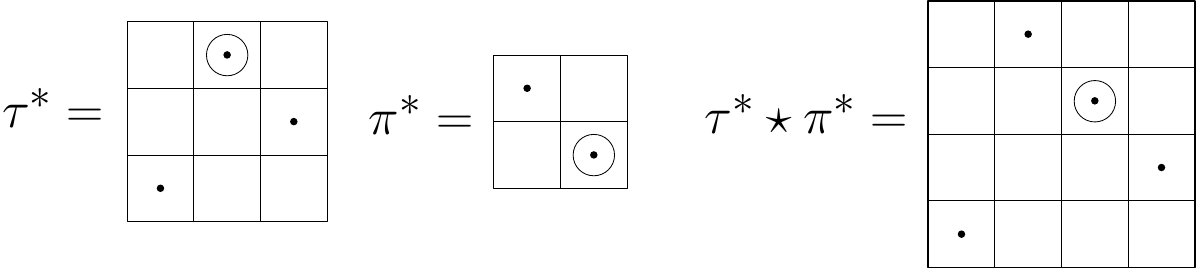} \, .
\end{equation*}
\end{defin}

\begin{rem}\label{rem:integerdomain}
Note that if $\pi^* \star \tau^* = \pi^* \star \sigma^*$, then $\tau^* = \sigma^* $.
Similarly, if $\tau^* \star \pi = \sigma^* \star \pi^*$, then $\tau^* = \sigma^* $.
\end{rem}

It is straightforward to observe that this is an associative presheaf with $\mathtt{MPer}[\emptyset] $, where the unit is the unique marked permutation on $\emptyset $, denoted by $\bar{1}$.
We call it the \textit{presheaf of marked permutations}.
Hence, from \cref{thm:Hopffunctor}, the algebra $\mathcal{A}(\mathtt{MPat})$ is indeed a Hopf algebra.

\subsection{Strategy for establishing the freeness of a pattern algebra\label{sec:strat}}

We now discuss the general strategy that we employ when establishing the freeness of a pattern Hopf algebra.
In particular, we clarify what is the relation between unique factorization theorems and freeness of the algebra of interest.
Let $\mathcal{S}\subseteq \mathcal{G}(h)$ be a collection of objects in a presheaf $h$.
Then the set $ \{\pat_s |  s\in \mathcal{S} \} $ is a set of free generators of $\mathcal{A}(h)$ if the set
$$ \left\{\prod_{s\in S} \pat_s \Big|  S \text{ multiset of elements of } \mathcal{S} \right\}\, , $$
is a basis of $\mathcal{A}(h)$.
This is usually established by connecting this set with the set $\left\{p_a | a \in \mathcal{G}(h)\right\}$, which is known to be a basis by \cref{rem:lipatfunc}.
This connection is done with the following ingredients:

\begin{itemize}
\item An order $\preceq$ in $\mathcal{G}(h) $.

\item A bijection $f$ between $\{ S \text{ multiset of elements of } \mathcal{S} \}$ and $ \mathcal{G}(h) $, which is usually phrased in terms of a \textit{unique factorization theorem}.
See for instance \cref{thm:Lyndonmperfact}.

\item The property that, for any $S$ multiset of $\mathcal{S}$,
\begin{equation}\label{eq:product}
 \prod_{s\in S} \pat_s = \sum_{t \preceq f(S)} c_{t, S} \pat_t \, ,
\end{equation}
with non-negative coefficients $c_{t, s}$ such that $c_{f(S), S}\neq 0$.
\end{itemize}

These are enough to establish the desired freeness.
In the commutative presheaf case, the unique factorization theorem is the one that we naively would expect; see \cref{thm:comuUFT}.
For the case of presheaf on permutations and the presheaf on marked permutations, the construction of the set $\mathcal{S}$ is more technical.

\begin{conj}\label{conj:freeness}
The pattern Hopf algebra of any associative presheaf is free.
\end{conj}

\section{Substructure algebras}
%\raul{use ast notation for presheaf prod, cdot for G prod}

%We will discuss presheafs on sets in this section.
In this section we describe the construction of the pattern Hopf algebra mentioned in \cref{thm:conHopfalgebra}.
We establish that for a combinatorial presheaf $h$, $\mathcal{A}(h)$ is an algebra; see \cref{thm:algfunctor}.
Moreover, if $h$ is connected and endowed with an associative structure, then $\mathcal{A}(h)$ is a Hopf algebra; see \cref{thm:conHopfalgebra}.
We also describe the space of primitive elements in $\mathcal{A}(h)$, clarify that $\mathcal{A}$ is in fact functorial; see \cref{thm:functoriality}, and find some identities and properties of the pattern functions in \cref{prop:Magnum,prop:representable}.

\subsection{Coinvariants and pattern algebras}

Given a combinatorial presheaf $h$, let $[n]=\{1, \dots , n\} $ for $n$ non-negative integer.
Then, recall that we define the \textit{coinvariants} of $h$ as the family
$$\mathcal{G}(h) = \biguplus_{n\geq 0} h[n]_{\sim } \, .$$

For an object $a\in \biguplus_J h[J] $, we defined the \textit{pattern function} above in \cref{defin:pattern}.
We have the following observation, that allows us to define pattern functions on equivalence classes.

\begin{obs}\label{prop:welldef}
This definition only depends on the isomorphism type of $a, b$.
That is, if $a_1\in h[I_1], a_2\in h[I_2], b_1\in h[J_1], b_2\in h[J_2]$ are so that $a_1\sim a_2, b_1\sim b_2$ then $\pat_{a_1}(b_1) = \pat_{a_2}(b_2)$.
\end{obs}

This allows us to define the pattern functions $ \pat_a $ for $a \in \mathcal{G}(h)$ as function in $ \mathcal{F}(\mathcal{G}(h), k)$.

%\begin{proof}
%Because $a_1 \sim a_2$, we have that $$ \{J' \subseteq J_1  \text{ s.t. } b_1|_{J'} \sim a_1 \} = \{J' \subseteq J_1  \text{ s.t. } b_1|_{J'} \sim a_2 \}\, , $$
%so $\pat_{a_1}(b_1) = \pat_{a_2}(b_1)$.
%It remains to prove that $\pat_{a_2}(b_1) = \pat_{a_2}(b_2)$.
%
%Because $b_1 \sim b_2$, there exists some bijective map $f:J_1 \to J_2$ such that $h[f](b_2) = b_1$.
%This bijection lifts to a bijection between $\{\tilde{J} \subseteq J_2  \}$ and $\{J' \subseteq J_1  \}$, via $\tilde{J} \mapsto f^{-1}(\tilde{J})$.
%
%We claim that this in fact restricts to a bijection between $\{\tilde{J} \subseteq J_2  \text{ s.t. } b_2|_{\tilde{J} } \sim a_2 \}$ and $\{J' \subseteq J_1  \text{ s.t. } b_1|_{J'} \sim a_2 \}$.
%Indeed, let $\tilde{J} \subseteq J_1$, let $\inc_{J_2, \tilde{J}}: \tilde{J} \to J_2$, $ \inc_{J_1, f^{-1}(\tilde{J})} $  denote the inclusion maps of $\tilde{J}$ in $J_2$ and $f^{-1}(\tilde{J} ) $ in $J_1$, respectively.
%Then
%$$b_1|_{f^{-1}(\tilde{J}) } = h[f^{-1} \circ \inc_{J_2, \tilde{J}}](b_2) = h[\inc_{J_1, f^{-1}(\tilde{J})} \circ f^{-1}](b_2) = h[f^{-1}](b_2|_{\tilde{J}})\, , $$
%so $ b_2|_{\tilde{J}} \sim b_1|_{f^{-1}(\tilde{J}) } $.
%This proves that $\pat_{a_2}(b_1) = \pat_{a_2}(b_2)$.
%\end{proof}

Recall that we write 
$$\mathcal{A}(h) : = \spn \{ \pat_a | \, a \in \mathcal{G}(h) \} \subseteq \mathcal{F}(\mathcal{G}(h), k )\, , $$
for the linear space spanned by all pattern functions.

\begin{rem}\label{rem:lipatfunc}
If two coinvariants $a, b$ are such that $|a| \geq |b|$ and $a\neq  b$, then $\pat_a (b)=0$.
We also have $\pat_b(b) = 1$.
Hence, the set $\{ \pat_a | a\in \mathcal{G}(h) \}$ is a basis of $\mathcal{A}(h)$.%:
\end{rem}

For $a, b $ objects and $c \in h[C]$, we defined the \textit{quasi-shuffle} number as follows:
\begin{equation}\label{eq:covernmr}
\begin{split}
\binom{c}{a, b} = \left| \{(I, J) \, \text{ such that } \, \,  I \cup J = C \, ,\, \, c|_{I} \sim a, \, c|_{J} \sim b \} \right| \, .  
\end{split}
\end{equation}
And a \textit{quasi-shuffle} as a pair that contributes to the coefficient above.
This is invariant under the equivalence classes of $\sim $, as it can be show in a similar way to \cref{prop:welldef}.

In \cref{thm:algfunctor} we observe that $\mathcal{A}(h) $ is a subalgebra of $\mathcal{F}(\mathcal{G}(h), k)$ with the pointwise multiplication structure and unit $ \sum_{c \in h[\emptyset ]} \pat_c $.
We present now the proof.

\begin{proof}
Fix $x\in h[I]$, and note that $\pat_a(x) \pat_b(x) $ counts the following
\begin{equation}
\begin{split}
\pat_a(x) \pat_b(x) &= \left| \{A \subseteq I \text{ s.t. } x|_A \sim a \} \times \{B \subseteq I \text{ s.t. } x|_B \sim b \}\right| \\
				   &= \left| \{(A, B) \text{ s.t. } A ,  B \subseteq I,  \, x|_A \sim a, \,  x|_B \sim b \} \right| \\	
				   &= \sum_{C\subseteq I} \left| \{(A, B) \text{ s.t. } A \cup B = C,  \, x|_A \sim a, \,  x|_B \sim b,   \} \right| \\
				   &= \sum_{C\subseteq I} \binom{x|_C}{a, b} = \sum_{c\in \mathcal{G}(h)} \binom{c}{a, b}\pat_c(x) \, .
\end{split}
\end{equation}

Hence, the space $\mathcal{A}(h) $ is closed for the product of functions.
Further, it is easy to observe that $\sum_{a \in h[\emptyset ] } \pat_a $ is the constant function equal to one, so this is a unit and $\mathcal{A}(h)$ is an algebra, concluding the proof.
\end{proof}

\begin{cor}[Products and quasi-shuffles]\label{cor:shufflesandquasishuffles}
Let $c\in h[I]$ and $b_i\in h[J_i]$ for $i=1, \dots , k$, and let 
$$\binom{a}{b_1, \dots , b_k} \coloneqq \left|\left\{(J_1, \dots , J_k) \text{ s.t. } \bigcup_{i=1}^k J_i = I, \, \, a|_{J_i} \sim b_i \, \, \forall i=1, \dots , k \right\} \right| \, .$$

Then we have that
$$ \prod_{i=1}^k \pat_{b_i} = \sum_c \binom{c}{b_1, \dots , b_k} \pat_c   \, . $$
\end{cor}

\begin{defin}[Shuffles and quasi-shuffles]\label{defin:qsands}
If an object $a$ is such that $\binom{a}{b_1, \dots , b_k}  >0 $ we say that $a$ is a quasi-shuffle of $b_1, \dots , b_k$.
In addition, if $|a| = \sum_{i=1}^k |b_i|$, we say that $a$ is a \textit{shuffle} of $b_1, \dots, b_k$.
\end{defin}

\subsection{Coproducts on pattern algebras}

In this section we consider an associative presheaf $(h, \ast, 1)$.
Concretely, our combinatorial presheaf $h$ is endowed with an associative product $\ast $ and a unit $1 \in h[\emptyset ]$.

\begin{defin}[Product structure in $\mathcal{G}(h)$]\label{defin:prodonG}
If $(h, \ast, 1)$ is an associative presheaf, then $\mathcal{G}(h)$ inherits an associative product.
If $a $ is an object, we denote its equivalence class under $\sim $ by $\bar{a}$ in this remark.
The associative product in $\mathcal{G}(h)$ is defined as follows:

Let $a \in h[n_1], b \in h[n_2] $ and denote $[n_1+1, n_1 + n_2] = \{ n_1 +1 , \dots , n_1 + n_2 \}$.
Consider $st $ the order preserving map $st : [n_1+1, n_1 + n_2] \to [n_2]$, and let $b'  = h[st](b)$.
Then we define the product in $\mathcal{G}(h)$ as $\bar{a}\cdot \bar{b} : = \overline{a \ast b'} \in h[n_1 + n_2]_{\sim } $.

It is a direct computation to see that $\overline{a\ast b}$ does not depend on the representative chosen for $\overline{a}$ and $\overline{b}$.
Thus we have a well defined operation in $\mathcal{G}(h)$.
\end{defin}

Recall that in \cref{eq:coprodformula}, we defined the coproduct in $\mathcal A (h)$ as 
$$\Delta  \pat_a := \sum_{\substack{b, c\in \mathcal{G}(h) \\  b \cdot c = a }} \pat_b \otimes \pat_c \, .$$
We also introduce the counit $\epsilon: \mathcal{A}(h) \to k$ that sends $\pat_a $ to $\mathbb{1}[a = 1]$, where $1$ stands for the unit in the associative presheaf $h$.
%\begin{defin}
%Let $a\in \mathcal{G}(h)$.
%Then, define 
%
%Note that the right hand side is a finite sum, because $b \cdot c = a$ implies that $ |b|, |c| \leq |a|$, so this is well defined.
%Define further the map 
%\end{defin}

We recall and prove \cref{thm:conHopfalgebra} here.
%
%\begin{thm}[The pattern Hopf algebra]\label{thm:Hopffunctorp2}
%Let $(h, \ast, 1) $ be an associative presheaf.
%Then, the maps $\Delta $ and $\epsilon $ give $\mathcal A (h)$ a structure of a coalgebra.
%Furthermore, together with pointwise multiplication of functions, this defines a bialgebra structure in $\mathcal{A}(h)$.
%
%Further, $\mathcal{A}(h)$ is a Hopf algebra whenever $h$ is a connected presheaf.
%\end{thm}

Remark that we refer to the pattern Hopf algebra of $(h, \ast, 1)$ simply by $\mathcal{A}(h)$ instead of $\mathcal{A}(h, \ast, 1)$, for simplicity of notation, whenever emphasis on the role of $\ast $ is not needed.

\begin{proof}[Proof of \cref{thm:conHopfalgebra}]
First, we note that $\Delta $ is trivially coassociative, from the associativity axioms of $\ast $ described in \cref{obs:naturality}.
That $\epsilon $ is a counit follows from the unit axioms on $(h, \ast , 1)$.

We first claim that, for $a, x, y\in \mathcal{G}(h)$,
\begin{equation}\label{eq:coprodprop}
\Delta \pat_a (x,  y ) = \pat_a (x\cdot y ) \, ,
\end{equation}
using the natural inclusion $\mathcal{F}(\mathcal{G}(h), k)^{ \otimes 2} \subseteq \mathcal{F}(\mathcal{G}(h)^2, k)$.

Indeed, take representatives $x\in~h[n_1]$ and $y\in~h[n_2]$ with no loss of generality, and write $B=[n_1]$, $C = \{n_1+1, \dots , n_1+n_2\}$.
Let $st $ be the order preserving map between $C$ and $[n_2]$.
Then
\begin{equation}
\begin{split}
\pat_a (x\ast y ) &= \left| \{J \subseteq B\sqcup C  \text{ such that }(x\ast y)|_J \sim a   \} \right|  \\
						   &=   \left| \{J \subseteq B\sqcup C  \text{ such that } x|_{J\cap B} \ast y|_{J\cap C} \sim a   \} \right|  \\
						   &=  \sum_{\substack{b, c\in \mathcal{G}(h) \\ a \sim b \cdot c }} \left| \{J \subseteq B\sqcup C \text{ such that } x|_{J\cap B} \sim b , \, \, y|_{J\cap C} \sim c   \} \right|  \\
						   &=   \sum_{\substack{b, c\in \mathcal{G}(h) \\ a \sim b \cdot c }} \left| \{J \subseteq B  \text{ s.t. } x|_J \sim b \} \times  \{ J \subseteq C  \text{ s.t. }  y|_{J} \sim c   \} \right|  \\
						   &= \sum_{\substack{b, c\in \mathcal{G}(h) \\ a \sim b \cdot c }} \pat_{b}(x)\pat_{c}(y) = \Delta \pat_a(x, y)   \, .
\end{split}
\end{equation}

Since both functions take the same values on $\mathcal{G}(h)^2 $, we conclude that \eqref{eq:coprodprop} holds.

The following are the bialgebra axioms that we wish to establish:
\begin{align*}
\Delta(\pat_a \pat_b ) &= \Delta(\pat_a ) \Delta(\pat_b ) \, , \\
\Delta \left( \sum_{a \in h[\emptyset ]} \pat_a \right) &=  \left( \sum_{a \in h[\emptyset ]} \pat_a \right) \otimes \left( \sum_{a \in h[\emptyset ]} \pat_a \right) \, , \\
\epsilon (\pat_a \pat_b ) &= \epsilon ( \pat_a ) \epsilon (\pat_b)  \, , \\
\epsilon \left( \sum_{a \in h[\emptyset ]} \pat_a \right) &= 1\, .
\end{align*}
The last three equations are direct computations.
For the first equation, we use \eqref{eq:coprodprop} as follows:
take $a, b, x, y\in \mathcal{G}(h)$, then 
$$\Delta(\pat_a \pat_b )(x, y) = (\pat_a \pat_b )(x \cdot  y) =  \pat_a ( x \cdot y)  \pat_b ( x \cdot y) = ( \Delta\pat_a  \Delta\pat_b) (x, y)  \, . $$
This concludes the first part of the proof.

Now suppose that $h$ is connected, so that the zero degree component $\mathcal{A}(h)_0$ is one dimensional.
From \cite{takeuchi71}, because $\mathcal{A} ( h ) $ is commutative, it suffices to establish that the group-like elements of  $\mathcal{A} ( h ) $ are invertible.
Now it is a direct observation that any group-like element is in $\mathcal{A}(h)_0$, which is a one dimensional algebra, so all non-zero elements are invertible.
This concludes that $\mathcal{A}(h)$ is a Hopf algebra.
\end{proof}

\subsection{The space of primitive elements\label{sec:primelem}}

In this section we give a description of the primitive elements of any pattern Hopf algebra.
Denote by $P_H \coloneqq \{a \in H | \, \Delta a = a\otimes 1 + 1 \otimes a \}$ the space of \textit{primitive elements}, where $1$ stands for the unit of the Hopf algebra.

\begin{defin}[Irreducible objects]\label{defin:indec}
Let $(h, \ast, 1)$ be a connected associative presheaf.
An object $t\in h[I] $ with $t\neq 1 $ is called \textit{irreducible} if any two objects $a \in h[A], b \in h[B]$ such that $a \ast b = t $ and $A\sqcup B = I$ have either $a = 1$ or $b = 1$.

The notion of irreducibility lifts to $\mathcal{G}(h)$.
That is, a coinvariant $t \in\mathcal{G}(h)$ with $t \neq 1$ is \textit{irreducible} if any $a, b\in\mathcal{G}(h)$ such that $a \cdot b = t $ have either $a = 1$ or $b = 1$.
We have that $a\in h[I]$ is irreducible if and only if the corresponding equivalence class $\bar{a} \in \mathcal{G}(h)$ is irreducible.
The family of irreducible equivalent classes in $\mathcal{G}(h)$ is denoted by $\mathcal{I} (h) \subseteq \mathcal{G}(h)$.
\end{defin}

\begin{prop}[Primitive space of pattern Hopf algebras]\label{prop:primit}
Let $(h,  \ast, 1_h)$ be a connected associative presheaf, and $\mathcal{I}(h)$ the set of irreducible elements in $\mathcal{G}(h)$.
Then the primitive space is given by:

$$P_{\mathcal{A}(h)} = \spn \{\pat_a | a \in \mathcal{I}(h)\} \, . $$
\end{prop}

\begin{proof}
If $f$ is irreducible, it is straightforward to observe that $\pat_f$ is primitive.
On the other hand, let $a=\sum_{f\in \mathcal{G}(h)} c_{f } \pat_{f }$ be a generic primitive element from $\mathcal{A}(h)$.
Then, the equation $\Delta a = a \otimes \pat_{1_h} + \pat_{1_h} \otimes \, a $ becomes
$$\sum_{g_1, g_2 \in \mathcal{G}(h) } c_{g_1 \cdot g_2 } \pat_{g_1}\otimes \pat_{g_2 } = \sum_{f} c_{f} ( \pat_{f} \otimes \pat_{1_h} + \pat_{1_h} \otimes \pat_{f } ) \, . $$

From \cref{rem:lipatfunc}, $\{\pat_{g_1} \otimes  \pat_{g_2 }\}_{g_1, g_2\in \mathcal{G}(h) } $ is a basis of $\mathcal{A}(h)^{\otimes 2}$, so we have that for any $g_1\neq 1_h, g_2 \neq 1_h$,
$$c_{g_1\cdot  g_2 } = 0 \, .  $$
Thus we conclude that $a$ is a linear combination of the set $\left\{ \pat_{f} |  f \in \mathcal{I}(h) \right\}$, as desired.
\end{proof}

\subsection{The pattern algebra functor}

In this section, we see that the mapping $\mathcal{A}$ is in fact functorial, bringing the parallel between the pattern Hopf algebras and the Fock functors even closer.

\begin{thm}[Pattern algebra maps]\label{thm:functoriality}
If $f: h\Rightarrow j $ is a morphism of combinatorial presheaves, then the following formula

\begin{equation}\label{eq:functordefin}
\mathcal{A}[f](\pat_a) := \sum_{f(b) = a } \pat_b \in \mathcal{A}(h)\, ,
\end{equation}
defines an algebra map $\mathcal{A}[f]: \mathcal{A}(j) \to \mathcal{A}(h)$.

Further, if $f$ is a morphism of associative presheaves, then $\mathcal{A}[f]$ is a bialgebra morphism.
Consequently, if $h, j$ are connected, this is a Hopf algebra morphism.
\end{thm}

The sum in \eqref{eq:functordefin} is finite, so this functor is well defined.

\begin{proof}
The map $\mathcal{A}[f]$ is linear and sends the unit $\sum_{a \in j[\emptyset ] } \pat_a $ of $\mathcal{A}(j)$ to 
$$\sum_{a \in j[\emptyset ] } \sum_{\substack{b \in h[\emptyset ] \\ f(b) = a}} \pat_b = \sum_{b \in j[\emptyset ] } \pat_b  \, . $$
Hence, to establish that $\mathcal{A}[f]$ is an algebra morphism, it suffices to show that it preserves the product on the basis, \textit{i.e.} that
$\mathcal{A}[f](\pat_a \pat_b) = \mathcal{A}[f](\pat_a ) \mathcal{A}[f](\pat_b )$.

It is easy to see that this holds if we have that
$$\binom{f(c')}{a, b} = \sum_{\substack{a', b'\in \mathcal{G}(h) \\ f(a') = a, f(b')=b}} \binom{c'}{a', b'} \, , $$
for any $a, b \in \mathcal{G}(j), c'  \in \mathcal{G}(h)$.

Indeed, if $a\in j[A], b\in j[B] $ and $c'\in h[C]$, then for any set $I$, by naturality of $f$, we have $f(c')|_I = f(c'|_I)$.
Then 
\begin{align*}
\binom{f(c')}{a, b}&= \left| \Big\{ (I, J) \text{ s.t. }  f(c')|_I \sim a, \, f(c')|_J\sim b, \, I\cup J = C\Big\} \right| \\
				  &= \sum_{\substack{a', b' \in \mathcal{G}(h) \\ f(a') = a, \, f(b') = b} } \left|\Big\{ (I, J)   \text{ s.t. }   c'|_I \sim a', \, c'|_J\sim b', \, I\cup J = C \Big\} \right| \\
				  &= \sum_{\substack{a', b' \in \mathcal{G}(h) \\ f(a') = a, \, f(b') = b} }\binom{c'}{a', b'}  \, .
\end{align*}

Now suppose further that $f$ is an associative presheaf morphism between $(h, \ast_h, 1_h)$ and $(j, \ast_j, 1_j)$.
That $\mathcal{A}[f]$ preserves counit follows from $f(1_h) = 1_j$.
That $\mathcal{A}[f]$ preserves $\Delta $ follows because both $\mathcal{A}[f]^{\otimes 2} (\Delta \pat_a) $ and $\Delta (\mathcal{A}[f] \pat_a ) $ equal
$$\sum_{\substack{b', c' \in \mathcal{G}(h)\\ f(b' \cdot_h c') = a}} \pat_{b'} \otimes \pat_{c'} \, , $$
since $f(b' \ast_h  c') = f(b') \ast_j f( c')$, concluding the proof.
\end{proof}

\begin{defin}[The pattern algebra functor]
Because of \cref{thm:functoriality}, we can define the \textit{pattern algebra} contravariant functor $\mathcal{A}:\mathtt{CPSh} \to \mathtt{Alg}_{k}$.

This functor when restricted to the subcategory of associative presheaves is also a functor to bialgebras as $\mathcal{A}:Mon(\mathtt{CPSh}) \to \mathtt{BiAlg}_{k}$.
\end{defin}

\begin{smpl}[Graph patterns in permutations]
Consider again the associative presheaves $(\mathtt{Gr}, \oplus , \emptyset )$ on graphs and $(\mathtt{Per}, \oplus , \emptyset )$ on permutations.
The inversion graph of a permutation is a presheaf morphism $\mathtt{Inv}:\mathtt{Per} \Rightarrow \mathtt{Gr}$ defined as follows:
given a permutation $\pi = ( \leq_P, \leq_V)$ on the set $I$, we take the graph $\mathtt{Inv}(\pi)$ with vertex set $I$ and an edge between $i,  j\in I$, $i \neq j$ if 
$$i\leq_P j\Leftrightarrow j\leq_V i \, . $$

It is a direct observation that this map is indeed an associative presheaf morphism.
As a consequence, we have a Hopf algebra morphism $\mathcal{A}(\mathtt{Inv}):\mathcal{A}(\mathtt{Gr}) \to \mathcal{A}(\mathtt{Per})$.
\end{smpl}

%\raul{I did not investigate this but it may be the case that we can say something about the coradical filtration of Gr inside the one of Per}

\subsection{Magnus relations and representability}

The goal of the following two sections is to show that arithmetic properties of the permutation pattern algebra that are established in \cite{vargas14} are actually general properties of pattern algebras.

\begin{prop}[Magnus inversions]\label{prop:Magnum}
Let $h$ be a combinatorial presheaf, and consider the maps $M, N: \spn \mathcal{G}(h)  \to  \spn \mathcal{G}(h) $ given on the basis elements $a\in \mathcal{G}(h)$ by
$$M: a \mapsto \sum_{b\in \mathcal G (h) } \pat_{b }(a) b \, , $$
$$N: a \mapsto \sum_{b\in \mathcal G (h) } (-1)^{|a|+|b|} \pat_{b }(a) b \, , $$
and extended linearly to $\spn \mathcal{G}(h)$.
Then the maps $M, N$ are inverses of each other.
\end{prop}

This result was already known in the context of words in \cite{hoffman00} and \cite{aguiar10} and in the context of permutations in \cite{vargas14}.

\begin{proof}

We start by proving a relation on pattern functions.
Let $a\in h[I], c\in \mathcal G (h)$.
Then we claim that
$$ \sum_{b \in \mathcal{G}(h)} (-1)^{|b|} \pat_b(a)\pat_c(b) = (-1)^{|a|} \mathbb{1}[a \in c ] \, .$$

Indeed, for each pair of sets $(J, B) $ such that $J \subseteq B \subseteq I $ and $a|_J \in c$ it corresponds an object $b = a|_B$, and the patterns $B$ of $b$ in $a$, and $J$ of $c$ in $b$.
Observe that this is a bijective correspondence, so we have
$$ \sum_{b \in \mathcal{G}(h)} (-1)^{|b|} \pat_b(a)\pat_c(b) = \sum_{\substack{J \subseteq I\\ a|_J \in c}} \sum_{J\subseteq B \subseteq I } (-1)^{|B|} = (-1)^{|I|} \mathbb{1}[ a|_I \in c ] =  (-1)^{|a|} \mathbb{1}[a \in c ] \, .$$

It follows that
\begin{align*}
N(M(a)) =& N \left(\sum_{b \in \mathcal{G}(h)}  \pat_b(a ) b\right) = \sum_{b \in \mathcal{G}(h)} \pat_b(a) \sum_{c \in \mathcal{G}(h)} (-1)^{|b|+|c|} \pat_c(b) \\
	    =&  \sum_{c \in \mathcal{G}(h)} c (-1)^{|c|} \sum_{b \in \mathcal{G}(h)} (-1)^{|b|} \pat_b(a) \pat_c(b) \\
	    =& \sum_{c \in \mathcal{G}(h)} c (-1)^{|c|} (-1)^{|a|} \mathbb{1}[a= c ] = a \, ,
\end{align*}
and that
\begin{align*}
M(N(a)) =& M \left(\sum_{b \in \mathcal{G}(h)}  (-1)^{|a|+|b|} \pat_b(a ) b\right) = \sum_{b \in \mathcal{G}(h)} \pat_b(a)  \sum_{c \in \mathcal{G}(h)} (-1)^{|a|+ |b|} \pat_c(b) \\
	    =&  \sum_{c \in \mathcal{G}(h)} c (-1)^{|a|} \sum_{b \in \mathcal{G}(h)} (-1)^{|b|} \pat_b(a) \pat_c(b) \\
	    =& \sum_{c \in \mathcal{G}(h)} c (-1)^{|a|} (-1)^{|a|} \mathbb{1}[a= c ] = a \, ,
\end{align*}
as desired.
\end{proof}

\subsection{The Sweedler dual and pattern algebras}

Let $A$ be an algebra over a field $\mathbb{K}$.
Then the Sweedler dual $A^{\circ }$ is defined in \cite{sweedler69} as
$$A^{\circ } \coloneqq \{ g \in A^* | \ker g \text{ contains a cofinite ideal}  \} \, ,$$
where a \textit{cofinite ideal} $J$ of $A$ is an ideal such that $A / J $ is a finite dimensional vector space over $\mathbb{K}$.
There, it is established that $A^{\circ }$ is a coalgebra, where the coproduct map is the transpose of the product map.

Consider the following right action of $A $ on $A^*$: for $f\in A^*$ and $a, b\in A$, 
\begin{equation}\label{eq:actionsweedler}
(f \cdot b) (a) \coloneqq f(a b) \, .
\end{equation}

A description of $A^{\circ }$ is given in \cite[Proposition 6.0.3]{sweedler69}, as all \textit{representable} elements $f\in A^*$, that is all $f$ such that the vector space $\{f \cdot b\}_{b\in A}$ is finite dimensional.

Let $(h, \ast , 1)$ now be an associative presheaf, and consider the algebra generated by $\cdot $ in $\mathcal{G}(h)$:
$$\Alg ( h) \coloneqq \spn (\mathcal{G}(h), \cdot )\, . $$
Then $\mathcal{A}(h) \subseteq \Alg (h)^* = \mathcal{F} (\mathcal{G}(h), k)$.
In fact, the following proposition guarantees that we have $\mathcal{A}(h) \subseteq \Alg (h)^{\circ }$.
Observe that the coproduct in $\mathcal{A}(h)$ is precisely the transpose of the multiplication in $\spn (\mathcal{G}(h), \cdot )$, therefore $\mathcal{A}(h) \subseteq \Alg (h)^{\circ }$ is an inclusion of coalgebras.

\begin{prop}[Pattern algebra and the Sweedler dual]\label{prop:representable}
Let $h$ be an associative presheaf.
Then, its pattern algebra satisfies $\mathcal{A}(h) \subseteq \Alg (h)^{\circ }$.
\end{prop}

\begin{proof}
We claim that each pattern function is representable, which concludes the proof according to \cite[Proposition 6.0.3]{sweedler69}.
In fact, for $a, b, c\in \mathcal G (h)$:
\begin{align*}
\pat_a \cdot \, b \, (c) &=  \pat_a ( b \cdot c) = \Delta \pat_a (b, c)\\
				  &= \sum_{\substack{a_1, a_2 \in \mathcal G (h) \\ a=a_1 \cdot a_2 }}\pat_{a_1}(b) \pat_{a_2}(c) 
\end{align*}
That is, $\pat_a \cdot  \, b = \sum_{ a=a_1 \cdot a_2 }\pat_{a_1}(b) \pat_{a_2} $.
It follows that 
$$\spn \{\pat_a \cdot \, b\}_{b\in \mathcal{G}(h)} \subseteq \spn \{\pat_{a_2} \}_{\substack{a_2\in \mathcal{G}(h) \\|a_2| \leq |a|}}\, , $$
 which is finite dimensional.
\end{proof}

\section{Freeness of commutative presheaves\label{sec:comutfree}}

%\todo[inline]{indecomposable - irreducible}

We start this section with a discussion of factorization theorems for combinatorial presheaves.
We will observe that the factorization of an object in a connected associative presheaf into irreducibles is unique up to some possible commutativity.
This is a general fact for associative presheaves, and is a central point in establishing freeness of any pattern Hopf algebra so far in the literature.

We also dedicate some attention to commutative presheaves.
An almost immediate consequence of the general fact discussed above is that the pattern algebra of a commutative presheaf is free.

We also explore specific combinatorial presheaves that can be endowed with a commutative structure.
The main examples are graphs, already studied in \cite{whitney1932coloring}, marked graphs (see \cref{sec:mgraphpatalg}), set partitions (see \cref{sec:spartpatalg}), simplicial complexes and posets.

\subsection{Relations in general connected associative presheaves}

Consider a connected associative presheaf $(h, \ast, 1)$.
In this section we will not assume that $h$ is commutative.
Recall that for objects $a\in h[I], b\in h[J]$, if $A \subseteq I\sqcup J$ then $(a\ast b)|_A = a|_{A\cap I} \ast a|_{A\cap J}$, as described in \cref{obs:naturality}.
We recall as well that an object $t\in h[I]$ is called \textit{irreducible} if $t\neq 1$ and if $t = a \ast b$ only has trivial solutions.
%We define as well an irreducible coinvariant in $\mathcal G (h)$.

\begin{defin}[Set composition and set partition]\label{defin:setcomp}
Let $I$ be a finite set.
A \textit{set composition} of $I$ is a list $(B_1, \dots B_k)$, that can also be written as $B_1 | \dots | B_k$, of pairwise disjoint nonempty subsets of $I$, such that $I = \bigcup_i B_i$.
We denote by $\Pi_I$ the family of set compositions of $I$.
A \textit{set partition} of $I$ is a family $\boldsymbol{\pi } = \{I_1, \dots , I_k\}$ of pairwise disjoint nonempty sets, such that $\bigcup_i I_i = I$.
We write $\Sigma_I $ for the family of set partitions of $I$.

If $\opi \in \Pi_I$ is a set composition, we can define its underlying set partition of $I$ by disregarding the order of the list.
We denote it by $\boldsymbol{\lambda } (\opi )$.
\end{defin}

\begin{defin}[Factorization of objects and coinvariants]
Consider an associative presheaf $h$ that is connected, and an object $o\in h[I]$.
A factorization of $o$ is a word $(x_1, \dots , x_k)$ of objects such that $x_1 \ast \dots \ast x_k = o$.

A factorization $(x_1, \dots , x_k)$ of $o$ is a factorization \textit{into irreducibles} when each $x_i$ is an irreducible object for $i=1, \dots , l$.

A factorization of a coinvariant $a\in \mathcal{G}(h)$ is a decomposition of the form $a = s_1 \cdots s_k$.
This factorization is \textit{into irreducibles} if each $s_i$ is \textit{irreducible}.

It is clear that an object is irreducible if and only if its coinvariant is irreducible.
\end{defin}

To a factorization $(x_1, \dots , x_k)$ of $o\in h[I]$, there is a corresponding set composition $\opi = (I_1, \dots , I_k)  \models I$, where $x_i \in h[I_i]$.
This is indeed a set composition of $I$ by the definition of $\ast $.
This correspondence between factorizations and set compositions is injective.
Indeed, because the presheaf is connected, $o|_{A_i} = x_i $.

Conversely, not all set compositions yield a factorization.
The irreducible elements are precisely the ones where only the trivial set composition with one block yields a factorization.
%Note that any product of the form $x_1 \ast \ldots \ast x_l = o \in h[X]$ is a factorization: this product corresponds to a set composition $A_1 |  \dots | A_l $ of $X$, given by $A_i = \mathbb{X}(x_i)$.
%This set composition is indeed a factorization, as for any $i=1, \dots , l$ we have:
%
%\begin{align*}
%o|_{A_i} &=x_1|_{A_1\cap A_i} \ast \ldots \ast x_l|_{A_l\cap A_i} \\
%		&= x_1|_{ \emptyset} \ast \ldots \ast x_i|_{A_i} \ast \ldots \ast x_l|_{ \emptyset} \\
%		&= 1 \ast  \ldots  \ast x_i \ast  \ldots  \ast 1 = x_i  \, ,
%\end{align*}
%thus it follows that $o = o|_{A_1} \ast \ldots \ast o|_{A_k} $.

%Recall that for a set composition $\opi$ of $I$, the underlying set partition of $I$ is denoted by $\boldsymbol{ \pi }(\opi)$.

\begin{thm}\label{thm:comuUFT}
Let $h$ be an associative presheaf and $o \in h[I]$ an object.
If $\opi_1, \opi_2$ are factorizations into irreducibles of $o$, then their underlying set partitions $\boldsymbol{\lambda }(\opi_1)$ and $\boldsymbol{\lambda }(\opi_2)$ are the same.
\end{thm}

In particular, the number of irreducible factors $j(o)$ and the multiset of irreducible factors $\mathfrak{fac}(o)$ of an object are well defined and do not depend on the factorization into irreducibles at hand.

\begin{proof}
Suppose that $o$ has two distinct factorizations 
\begin{equation}\label{eq:twofactorizationsofa}
 o = l_1 \ast  \ldots \ast l_{k}  =  r_1 \ast \ldots \ast r_s  \, ,
\end{equation}
where $\opi_1 = A_1 | \dots | A_{k} $ and $\opi_2 = B_1 | \dots |  B_s $ are set compositions of $X$ such that $l_i = o|_{A_i}$ and $r_i = o|_{B_i}$.
Note that
$$l_j = o|_{A_j} = r_1|_{B_1\cap A_j} \ast \ldots \ast r_s|_{B_s\cap A_j} \, .$$

Because $l_j$ is irreducible, for each $j$ there is exactly one $i$ such that $B_i\cap A_j \neq \emptyset $, so $A_j \subseteq B_i$, and $\boldsymbol{ \tau }(\opi_1) $ is coarser than $\boldsymbol{ \tau }(\opi_2) $.
By a symmetrical argument, we obtain that $\boldsymbol{ \tau }(\opi_1) $ is finer than $\boldsymbol{ \tau }(\opi_2) $, so we conclude that $\boldsymbol{ \tau }(\opi_1) = \boldsymbol{ \tau }(\opi_2) $.
It follows that the number of factors and the multiset of factors are well defined.
\end{proof}

This implies the following for factorizations on $\mathcal{G}(h)$:

\begin{cor}\label{cor:factsthm}
Consider an associative presheaf $(h, \ast, 1)$, together with the usual product in $\mathcal{G}(h)$, denoted by  $\cdot $.
Consider also $a \in \mathcal{G}(h)$.
If $a=x_1\cdots x_l = s_1 \cdots s_s $ are two factorizations into irreducibles, then the multisets $\{x_1, \dots ,  x_l\}$ and $\{s_1,  \dots , s_t\}$ coincide.
\end{cor}

In particular, the number of irreducible factors $j(a)$ and the multiset of irreducible factors $\mathfrak{fac}(a)$ of a coinvariant are well defined and do not depend on the factorization into irreducibles at hand.

Given an alphabet $\Omega $, denote the set of words on $\Omega $ by $\mathcal{W}(\Omega )$.

\begin{prob}[Factorization theorems in associative presheaves]\label{prob:fibers}
Given an associative presheaf $(h, \ast, 1)$, describe $\mathcal{E}(h)$, the collection of fibers of the map
$$\Pi : \mathcal{W}(\mathcal{I}(h)) \to \mathcal{G}(h)  \, , $$
defined by taking the product of the letters of a word in $\mathcal G (h)$.
\end{prob}

The fibers of this map, that is, the sets of words $\Pi^{-1}(a)$, correspond to the different factorizations of an object $a \in \mathcal{G}(h)$.
In general, according to \cref{cor:factsthm}, a fiber consists of a set of words of irreducible elements that result from one another by permuting its letters.
Whenever $h$ is a commutative presheaf, \cref{cor:UFTcommut} tells us that the fibers are as big as possible, restricted to \cref{cor:factsthm}.
This means that, in this case, each fiber is a set of words resulting from a permutation of a word in $\mathcal{W}(\mathcal{I} (h) )$.

Take the example of  the combinatorial presheaf $\mathtt{Per}$, where no non-trivial rearrangement of the irreducible factors of a factorization of an object yields a distinct factorization of the same object.
In this example, the fibers $\Pi^{-1}(a)$ are singletons.
In the case of marked permutations, in \cref{thm:isomonoids} we show that only transpositions of specific irreducible marked permutations remain factorizations of the same coinvariant $a$.

\begin{rem}\label{rem:vargasispowerful}
It can be seen that the freeness proof in \cref{thm:comutUFTfree} depends solely on the corresponding unique factorization theorem, that is, on $\mathcal{E} (h)$.
That is also the case in the proof given in \cite{vargas14} for the presheaf on permutations, and the proof below for marked permutations.

This motivates \cref{conj:freeness}, as it seems that the freeness of the pattern Hopf algebra only depends on the description of the fibers $\mathcal{E} (h)$.
In this way, for instance, we can immediately see that the pattern Hopf algebra $\mathcal A (\mathtt{SComp})$, defined below, is free.
This follows because it has a unique factorization theorem of the type of the one in the associative presheaf on permutations.
\end{rem}

\subsection{Proof of freeness on commutative presheaves}

Recall that a commutative presheaf is an associative presheaf $(h, \ast, 1)$ such that $\ast_{A, B} = \ast_{B, A} \circ \text{ twist}_{A, B} $; see \cref{def:asspresheaf}.
In the case of commutative presheaves, we have that any rearrangement of a factorization of an object $o$ yields another factorization of $o$.
For this reason, in the context of commutative presheaves, a set partition is also referred to as a factorization, and we have the following.

\begin{cor}\label{cor:UFTcommut}
Let $(h, \ast, 1)$ be a connected commutative presheaf, and $a \in \mathcal{G}(h) $ an object.
Then, $a$ has a unique factorization into irreducibles $l_1, \dots , l_{j(a)} \in \mathcal{I}(h) $ up to commutativity of factors.
Equivalently, if $a \in h[X]$, there is a unique set partition that corresponds to a factorization of $o$ into irreducibles.
\end{cor}

\begin{thm}[Freeness of pattern algebras with commutative products]\label{thm:comutUFTfree}
Let $(h, \ast, 1)$ be a connected commutative presheaf.
Consider $\mathcal{I} (h)\subseteq \mathcal{G}(h) $ the family of irreducible elements of $h$.

Then $\mathcal{A}(h)$ is free commutative, and $ \{\pat_{\iota } | \,  \iota \in \mathcal{I} (h) \} $ is a set of free generators of $\mathcal{A}(h)$.
\end{thm}

\begin{proof}
We will show that the family
\begin{equation}\label{eq:basiscommutproof}
 \Bigg\{ \prod_{\iota \in \mathcal{L}} \pat_\iota \Big| \mathcal{L} \text{ multiset of elements in } \mathcal{I}(h) \Bigg\} \, ,
\end{equation}
is a basis for $\mathcal{A}(h)$.
The proof follows the strategy described in \cref{sec:strat}, by building an order $\leq $ in $\mathcal{G}(h) $ that is motivated in the unique factorization theorem in \cref{thm:comuUFT}.

Define the following partial strict order $<_p $ in $\mathcal{G}(h) $: we say that $\alpha <_p \beta $ if:
\begin{itemize}
\item $| \alpha  | < | \beta |$, or

\item $| \alpha  | = | \beta |$ and $j(\alpha ) < j(\beta)$.
\end{itemize}

In this way, $\leq_p$ is the order that we use to establish freeness.

Consider $\alpha \in \mathcal{G}(h)$ with $\alpha  = \iota_1  \cdots  \iota_{j(\alpha )} $ its unique factorization into irreducibles.
Then we claim
\begin{equation}\label{eq:triangcomu}
\prod_{i=1}^k \pat_{\iota_i} =\sum_{\beta \in \mathcal{G}(h)} \binom{\beta }{\iota_1, \dots, \iota_{j(\alpha )}} \pat_{\beta } =  \sum_{\beta \leq_p \alpha } c_{\beta } \pat_{\beta } \, ,
\end{equation}
where $c_a \geq 1$.
This implies that \eqref{eq:basiscommutproof} is a basis of $\mathcal{A}(h)$, and gives us the result.

Let us prove \eqref{eq:triangcomu}.
Pick representatives $l_i $ for $\iota_i$ such that $l_i \in h[A_i]$ and let also $a = l_1\ast \dots \ast l_{j(\alpha)}$ be an object.
Observe that the coinvariant of $a$ is $\alpha $.

First, observe that $a$ is a quasi-shuffle of $l_1, \dots , l_{j(\alpha)} $ by considering the patterns $A_1, \dots , A_{j(\alpha)}$; see \cref{rem:inflqshuf}.
Thus we have $c_{\alpha } \geq 1$.

To show that any term $\beta $ in \eqref{eq:triangcomu} with $c_{\beta } \neq 0$ has $ \beta \leq_p  \alpha $, consider the maximal $ \beta \in \mathcal G (h) $ that is a quasi-shuffle of $l_1, \dots , l_{j(\alpha)} $, and let $b$ be a representative in $\beta $,  such that $b\in h[Y]$.
By maximality, we have that $ \beta  \geq  \alpha$.
Our goal is to show that $b \sim a$.
Let $b= s_1 \ast \dots \ast s_{j(b)}$ be the unique factorization of $b$ into irreducibles, corresponding to the set partition $\{ C_1,  \dots ,  C_{j(b)} \}$.
Because $b$ is a quasi-shuffle of $l_1, \dots , l_{j(\alpha)} $, we can consider $B_1, \dots , B_{j(\alpha)} \subseteq Y $ sets such that $\bigcup_i B_i = Y$ and
$$b|_{B_i} \sim l_i  \, \, \, \text{ for } i=1, \dots , j(\alpha) \, . $$

In particular, we have that 
$$l_i \sim b|_{B_i} = s_1|_{C_1\cap B_i}\ast  \dots \ast s_{j(b)}|_{C_{j(b)}\cap B_i}\, , $$
 for $i=1 , \dots , j(\alpha)$.

By indecomposability of $l_i$, we have that, for each $i$, there is exactly one $j$ such that $C_{j}\cap B_i\neq \emptyset$.
From $\sqcup_j C_j = \bigcup_i B_i = Y$, we get that each $B_i$ is contained in some $C_j$.
So we can define a map $f:[j(\alpha)] \to [j(b)] $ such that $ B_i \subseteq C_{f(i)}$.
Thus, we have that $C_i = \bigcup_{k \in f^{-1}(i)} B_k $.

First observe that $|a|= \sum_i | A_i |= \sum_i |l_i | = \sum_i |  B_i| $ and $|b| = \sum_j | C_j | $ but $|b | = |Y| \leq \sum_j \sum_{i\in f^{-1}(j)} | B_i | = |a|$.
However, from $b \geq_p a$ we have that $|b| \geq |a|$ and thus we have $|b| = |a|$ and that the family $\{ B_1, \dots , B_{j(\alpha)} \}$ is disjoint.
Further, $f$ is a surjection, so we immediately have that $j(\alpha) \geq j(b)$ and because $b \geq_p a$, we must have an equality.
Thus, $f$ is a bijective map and we conclude that $C_{f(i)} = B_i $ for each $i=1, \dots , j(\alpha) = j(b)$.

We then conclude that 
$$ s_i = b|_{B_i} = b|_{C_{f(i)}} \sim  l_{f(i)}\, ,$$
and so, by commutativity, 
$$ b = s_1 \ast \dots \ast s_{j(b)} = l_1 \ast \dots \ast l_{j(\alpha)}  = a \, , $$
as desired.
\end{proof}

\subsection{Marked graphs}\label{sec:mgraphpatalg}

\begin{defin}[Marked graphs and two products]
For a finite set $I$, a marked graph $G^*$ on $I$ is a graph on the vertex set $I \sqcup \{*\}$.
This defines a combinatorial presheaf $\mathtt{MGr}$ via the usual notion of relabeling and induced subgraphs.

We can further endow the combinatorial presheaf $\mathtt{MGr}$ with two different associative presheaf structures.

First, the \textit{joint union}, $\vee $, which is defined as follows:
If $G_1^*\in \mathtt{MGr}[I], G_2^* \in \mathtt{MGr}[J]$ with $I\cap J = \emptyset$, then $G_1^*\vee_{I, J} G_2^*$ has no edges between $ I $ and $ J$, and the marked vertices are merged.

The second product, the \textit{inflation product} $ \star $ is defined as follows:
If $G_1^*\in \mathtt{MGr}[I], G_2^* \in \mathtt{MGr}[J]$ with $I\cap J = \emptyset$, then $G_1^*\star_{I, J} G_2^*$ results from  $G_1^*\vee_{I, J} G_2^*$ by adding the following edges: two vertices $i\in I, j\in J$ are connected in $G_1^*\star_{I, J} G_2^*$ if $i$ and $*$ are connected in $G_1^*$.
The unit of both products is the marked graph $1$ with a unique vertex and no edges.
\end{defin}

\begin{rem}
The graphs that are irreducible with respect to the $\vee $ product are the graphs $G^* $ such that the graph resulting from removing the marked vertex and its incident edges is a connected graph. In this case, we say that $G^* $ is $\vee$-connected.

In \cref{fig:mcongraphs} we have an example of a $\vee$ -connected marked graph and a $\vee$-disconnected marked graph.
\end{rem}

\begin{figure}[h]
\centering
\includegraphics[scale=0.7]{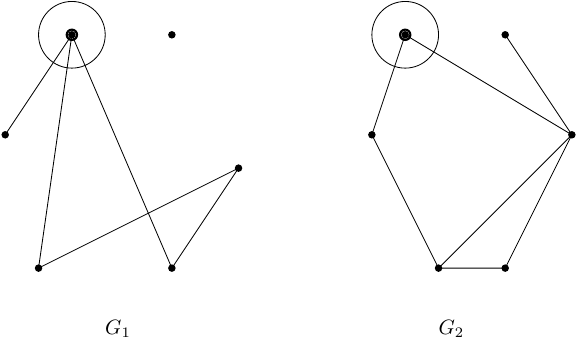}
\caption{A marked graph $G_1$ with three $\vee$-connected components and a $\vee$-connected graph $G_2$.\label{fig:mcongraphs}}
\end{figure}

Observe that $G_1^*\vee_{I, J} G_2^* = G_2^*\vee_{J, I} G_1^*$, so $\vee $ is a commutative operation, whereas $\star $ is not.
It follows from \cref{thm:comutUFTfree} that both $\mathcal{A}(\mathtt{MGr}, \vee, 1 ) $ and $\mathcal{A}(\mathtt{MGr}, \star, 1 ) $ are free algebras (their product structure is the same).
This is something that cannot be shown directly in $\mathcal{A}(\mathtt{MGr}, \star, 1 ) $.
Indeed, a unique factorization theorem on this associative presheaf has not yet been found, and only small irreducible marked graphs can be constructed; see \cref{fig:starindec}.
It is worthwhile to observe that the fibers under the map $\Pi $ described in \cref{prob:fibers} are non-trivial, as an example of a non-trivial relation can be seen in \cref{fig:starindecrel}

\begin{figure}[h]
\centering
\includegraphics[scale=0.4]{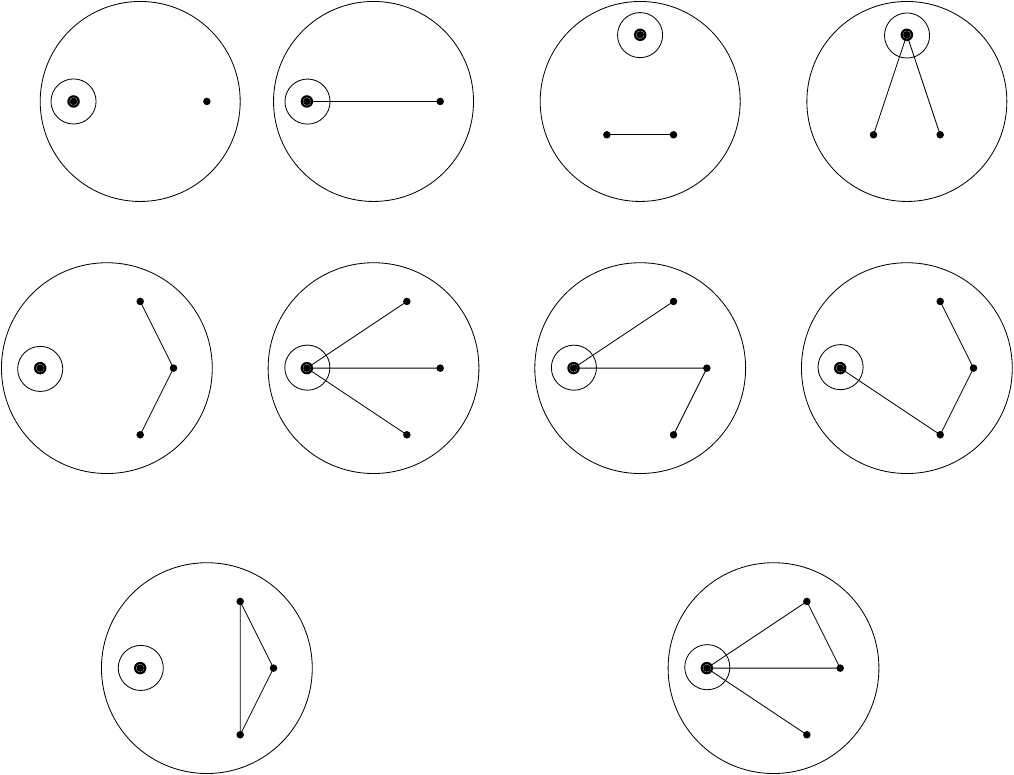}
\caption{The $\star$-irreducible marked graphs of size up to three.\label{fig:starindec}}
\end{figure}

\begin{figure}[h]
\centering
\includegraphics[scale=0.4]{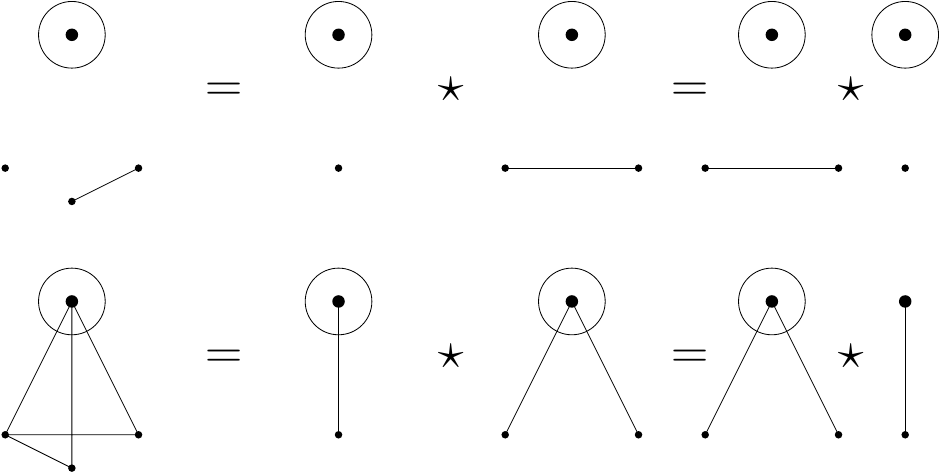}
\caption{The first relations between irreducible marked graphs over $\star $.\label{fig:starindecrel}}
\end{figure}

%Let us now study the pattern Hopf algebra of marked graphs under the associative product $\star $.
%It is still a fact that the algebraic structure of $\mathcal{A}(\mathtt{MGr}, \star , 1)$ is freely generated by $\pat_{G^*} $ for $G^* $ markedly connected graphs.
%However, the coalgebra structure is now described by the $\star $ product.
%In \cref{fig:starindec}, the first few irreducible marked graphs are described.
%The first observation is that a since the underlying associative presheaf is not commutative, a unique factorization theorem like \cref{cor:UFTcommut} cannot be aplied.

\subsection{Set partitions\label{sec:spartpatalg}}

We define here associative presheaves on set partitions $\mathtt{SPart}$ and set compositions $\mathtt{SComp}$, and show that $Sym $, the Hopf algebra of symmetric functions, is the pattern Hopf algebra on set partitions.
The functoriality of $\mathcal A $ also gives us a Hopf algebra morphism $Sym \to \mathcal A (\mathtt{SComp})$.

\begin{defin}[The presheaf on set partitions]
If $\boldsymbol{\pi }$ is a set partition of $I$ and $J \subseteq I$, then $\boldsymbol{\pi }|_J = \{I_1 \cap J, \dots , I_k \cap J\} $ is a set partition of $J$, after disregarding the empty sets.
This defines a presheaf structure $\mathtt{SPart}$ with $\mathtt{SPart}[I] = \Sigma_I$, the family of set partitions of $I$.

We further endow $\mathtt{SPart}$ with an associative structure $\uplus $ as follows: if $\boldsymbol{\pi } = \{I_1, \dots , I_q\}$, $\boldsymbol{\tau } = \{J_1, \dots , J_p\}$ are set partitions of the disjoint sets $I, J$, respectively, let $\boldsymbol{\pi } \uplus \boldsymbol{ \tau }  = \{I_1, \dots , I_q, J_1, \dots , J_p\}$ be a set partition of $I\sqcup J$.
It is straightforward to observe that $(\mathtt{SPart}, \uplus, \emptyset )$ is a commutative connected presheaf.
\end{defin}

Note that by \cref{thm:comutUFTfree}, the pattern Hopf algebra $\mathcal{A}(\mathtt{SPart})$ is free and the generators correspond to the irreducible elements of $\mathcal{G} (\mathtt{SPart} )$.
These correspond to the set partitions with only one block, and up to relabeling there is a unique such set partition of each size.
We write $\mathcal{I}(\mathtt{SPart}) = \{ \{[n]\} | n\geq 1\}$.

\begin{prop}
Let $\zeta : \mathcal{A} (\mathtt{SPart} ) \to Sym$  be the unique algebra morphism mapping  $\zeta : \pat_{ \{[n]\}} \mapsto p_n$, where $p_n$ is the power sum symmetric function.
This defines a Hopf algebra isomorphism.
\end{prop}

\begin{proof}
As $\zeta $ sends a free basis to a free basis, it is an isomorphism of algebras.
Furthermore, we observe that both $\pat_{ \{[n]\}} $ and $p_n$ are primitive elements in their respective Hopf algebras, so the described map is a bialgebra morphism.
Because the antipode is unique, it must send the antipode of $\mathcal{A} (\mathtt{SPart} )$ to the one of $ Sym$.
Thus, this is a Hopf algebra isomorphism.
\end{proof}

\begin{defin}[The presheaf on set compositions]
Let $I$ be a finite set, and recall the definition of a set composition in \cref{defin:setcomp}.
If $J \subseteq I$ and $\opi  = ( I_1, \dots , I_k  ) $ is a set composition of $I$, then $\opi|_J = ( I_1 \cap J, \dots , I_k \cap J ) $ is a set composition of $J$, after disregarding the empty sets.
This defines a presheaf structure $\mathtt{SComp}$ with $\mathtt{SComp}[I] = \Pi_I$, the family of set compositions of $I$.

We further endow $\mathtt{SComp}$ with an associative structure $\uplus $ as follows: if $\opi = (I_1, \dots , I_q)$, $\otau = (J_1, \dots , J_p)$ are set partitions of the disjoint sets $I, J$, respectively, let $\opi \uplus \otau  = (I_1, \dots , I_q, J_1, \dots , J_p)$ be a set composition of $I\sqcup J$.

It is straightforward to observe that $(\mathtt{SComp}, \uplus, \emptyset )$ is an associative connected presheaf.
Further, we can also observe that the map $\boldsymbol{\lambda } :\mathtt{SComp} \Rightarrow \mathtt{SPart}$ is an associative presheaf morphism.
\end{defin}

From the map $\boldsymbol{\lambda } :\mathtt{SComp} \Rightarrow \mathtt{SPart} $ we get a Hopf algebra morphism 
$$ \mathcal{A} (\boldsymbol{\lambda } ) :  Sym \to \mathcal{A}(\mathtt{SComp})  \, . $$

Observe that $ \mathcal{A}(\mathtt{SComp} )$ is a free algebra, because it has a unique factorization theorem of the same type of permutations under the $\oplus $ product, so according to \cref{rem:vargasispowerful} the proof in \cite{vargas14} holds also in this associative presheaf.
This is also the case for the well known Hopf algebra $QSym$, where it was established that it is free in \cite{hazewinkel01}, and when we regard both Hopf algebras as filtered Hopf algebras, the enumeration of generators for each degree coincide with the number of Lyndon words of a given size.

\begin{conj}[The algebraic structure of $ \mathcal{A}(\mathtt{SComp})$]\label{conj:QSym}
The Hopf algebra $ \mathcal{A}(\mathtt{SComp} )$ isomorphic to $QSym$.
\end{conj}

\section{Hopf algebra structure on marked permutations\label{sec:mperfree}}

In this section we consider the algebra structure of $\mathcal{A}(\mathtt{MPer} )$, and show that this pattern algebra on marked permutations is freely generated.
This will be done using a factorization theorem on marked permutations on the inflation product.

Our strategy is as follows: we describe a unique factorization of marked permutations with the inflation product, in \cref{cor:simpleUFT}.
This unique factorization theorem describes all possible factorizations of a marked permutation into irreducibles.

We further consider the Lyndon words on the alphabet of irreducible marked permutations, as introduced in \cite{chen58}.
This leads us to a notion of stable Lyndon marked permutations $\mathcal{L}_{SL}$, in \cref{defin:centralLyndon}.
Finally, we prove the following result, which is the main theorem of this section, as a corollary of \cref{thm:triang}:

\begin{thm}\label{thm:freemperm}
The algebra $\mathcal{A}(\mathtt{MPerm})$ is freely generated by $\{\pat_{\iota^*} | \iota^*\in \mathcal{L}_{SL} \}$.
\end{thm}

In the end of this section we compute the dimension of the space of primitive elements of the pattern Hopf algebra on marked permutations.
This, according to \cref{prop:primit}, can be done by enumerating the irreducible elements in $\mathcal{G}(\mathtt{MPer})$.

This section is organized as follows: we start in \cref{sec:UFT} and in \cref{sec:LFT} by establishing a unique factorization theorem in marked permutations.
%This is done by means of a correspondence between marked permutations and its factorization as words of indecomposable marked permutations.
%In \cref{thm:isomonoids}, we compute the fibers of this correspondence, and postpone its proof to \cref{sec:UFTProof}.
In \cref{sec:MTheorems} we state and prove the main theorem of this section.
The proofs of technical lemmas used in these sections are left to \cref{sec:proofoflemmas,sec:UFTProof}.
Finally, in \cref{sec:PrimEl}, we enumerate the irreducible marked permutations.

\subsection{Unique factorizations\label{sec:UFT}}

We work on the combinatorial presheaf of permutations $(\mathtt{Per}, \oplus, \emptyset)$ and on the combinatorial presheaf of marked permutations $(\mathtt{MPer}, \star, \bar{1})$, introduced in \cref{defin:per,defin:mper}.

\begin{defin}[Decomposability on the operations $\oplus$ and $\ominus$]
We say that a permutation is $\oplus$-indecomposable if it has no non-trivial decomposition of the form $\tau_1\oplus \tau_2 $, and $\oplus$-decomposable otherwise.
We say that a marked permutation is $\oplus$-indecomposable if it has no decomposition of the form $\tau \oplus \pi^* $ or $\pi^*\oplus \tau$, where $\pi^*$ is a marked permutation and $\tau$ is a non-trivial permutation, and $\oplus$-decomposable otherwise.
Similar definitions hold for $\ominus $.
%Similarly, we say that a permutation (resp. a marked permutation) is $\ominus$-indecomposable if it has no non-trivial decomposition of the form $\tau_1\ominus \tau_2 $ (resp. if it has no decomposition of the form $\tau \ominus \pi^* $ or $\pi^* \ominus \tau$, with $\tau $ non-trivial permutation), and $\ominus$-decomposable otherwise.
A permutation (resp. a marked permutation) is \textit{indecomposable} if it is both $\oplus$ and $\ominus$-indecomposable, and is \textit{decomposable} otherwise.
\end{defin}

We remark that a permutation is $\oplus$-indecomposable whenever it is irreducible on the associative presheaf $(\mathtt{Per}, \oplus, \emptyset)$ according to \cref{defin:indec}.
For marked permutations, \cref{defin:indec} specializes to the definition of irreducible marked permutations as follows:

\begin{defin}[Irreducible marked permutations]
%A permutation $\pi $ is called \textit{simple} if any inflation product of permutations that has $\tau [\tau_1, \dots , \tau_k] = \pi $ for $k>1$ satisfies $\tau = \pi$ and $\tau_1 = \dots = 1$.
A marked permutation $\pi^* $ is called \textit{irreducible} if any factorization $\pi^* = \tau^*_1 \star \tau_2^* $ has either $\tau^*_1 = \bar{1}$ or $\tau^*_2 = \bar{1}$.
\end{defin}

\begin{smpl}
Examples of irreducible marked permutations include $\bar{1}423, 23\bar{1}$ and $31\bar{4}2$; see \cref{fig:simplemps}.
These marked permutations are respectively an $\oplus$-decomposable, $\ominus$-decomposable and an indecomposable irreducible marked permutation.

\begin{figure}[h]
\centering
\includegraphics[scale=0.5]{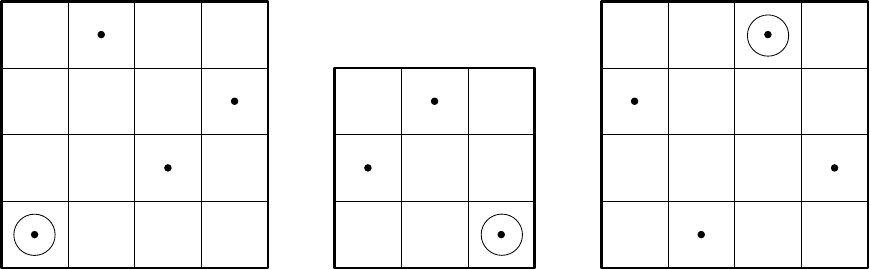}
\caption{Irreducible marked permutations\label{fig:simplemps}}
\end{figure}

\end{smpl}

\begin{rem}[Decomposable irreducible marked permutations]
If $\pi $ is an $\oplus$-indecomposable permutation, then $\bar{1}\oplus \pi $ and $\pi \oplus \bar{1} $ are irreducible marked permutations.
Similarly, if $\tau $ is an $\ominus$-indecomposable permutation, then $\bar{1}\ominus \tau $ and $\tau \ominus \bar{1} $ are irreducible marked permutations.

These are precisely the \textit{decomposable} (resp. $\oplus$-\textit{decomposable}, $\ominus$-\textit{decomposable}) \textit{irreducible marked permutations}.
\end{rem}

These decomposable irreducible marked permutations play an important role in the description of a free basis of $\mathcal{A}(\mathtt{MPer})$, because they are the only ones that get in the way of a unique factorization theorem for the inflation product.
In the following we carefully unravel all these issues.
%We discuss how to enumerate simple marked permutations in \cref{sec:PrimEl}.

\begin{rem}[$\oplus$-relations and $\ominus$-relations]\label{rem:oplusominusrelations}
Consider $\tau_1, \tau_2 $ $\oplus$-indecomposable permutations.
Then we have the following relations, called \textit{$\oplus$-relations}
$$ ( \bar{1}\oplus \tau_1 ) \star ( \tau_2 \oplus \bar{1} ) = ( \tau_2 \oplus \bar{1} ) \star ( \bar{1}\oplus \tau_1 ) =  \tau_2 \oplus \bar{1} \oplus \tau_1 \, . $$

Consider now $\pi_1, \pi_2 $ permutations that are $\ominus$-indecomposable.
Then we have the following relation, called \textit{$\ominus$-relations}
$$ ( \bar{1}\ominus \pi_1 ) \star ( \pi_2 \ominus \bar{1} ) = ( \pi_2 \ominus \bar{1} ) \star ( \bar{1}\ominus \pi_1 ) = \pi_2 \ominus \bar{1}\ominus \pi_1 \, . $$

We wish to establish in \cref{thm:isomonoids} that these generate all the inflation relations between irreducible marked permutations.
\end{rem}

We define the alphabet $\Omega : = \{\text{irreducible marked permutations} \}$, and consider the set $\mathcal{W}(\Omega ) $ of words on $\Omega$.
This set forms a monoid under the usual concatenation of words (we denote the concatenation of two words $w_1, w_2$ as $w_1 \cdot w_2$).
When $w\in \mathcal{W}(\Omega ) $, we write $w^*$ for the consecutive inflation of its letters.
So for instance $(\bar{1}2, \bar{2}1)^* = \bar{1}2 \star \bar{2}1 = \bar{2}13$.
We use the convention that the inflation of the empty word on $\Omega $ is $\bar{1}$. 
This defines the \textit{star map}, a morphism of monoids $ \star : \mathcal{W}(\Omega) \to \mathcal{G}(\mathtt{MPer})$.

For the sake of clarity, we avoid any ambiguity in the notation $a^*$ by using Greek letters for marked permutations, lowercase Latin letters to represent words on any alphabet, and upper case Latin letters to represent sets with an added marked element; see \cref{defin:mper}.

\begin{defin}[Monoidal equivalence relation on $\mathcal{W}(\Omega) $]
We now define an equivalence relation on $\mathcal{W}(\Omega) $.
For a word $w \in \mathcal{W}(\Omega ) $,  if $w = (\xi_1^* , \dots , \xi^*_k)$ is such that $\xi_i^* \star \xi_{i+1}^* = \xi_{i+1}^* \star \xi_i^* $ is an $\oplus$-relation or an $\ominus$-relation, we say that $w \sim (\xi_1^* , \dots \xi_{i-1}^*, \xi_{i+1}^*, \xi_i^*, \dots , \xi^*_k) $.

We further take the transitive and reflexive closure to obtain an equivalence relation on $\mathcal{W}(\Omega )  $.
\end{defin}

We trivially have that if $w_1 \sim w_2$ and $z_1 \sim z_2$, then $w_1\cdot z_1 \sim w_2\cdot z_2$.
This means that the set of equivalence classes $\mathcal{W}(\Omega)_{\sim}$ is a monoid.
Further, because of \cref{rem:oplusominusrelations}, the star map factors to a monoid morphism $ \mathcal{W}(\Omega)_{\sim } \to \mathcal{G}(\mathtt{MPer})$.

\begin{thm}\label{thm:isomonoids}
The star map $ \star : w \mapsto w^*$ defines an isomorphism from $ \mathcal{W}(\Omega) /_{ \sim }$ to the monoid of marked permutations with the inflation product.
\end{thm}

We postpone the proof of \cref{thm:isomonoids} to \cref{sec:UFTProof}, and explore its consequences here.

Informally, this theorem states that any two factorizations of a marked permutation $\pi^* $ into irreducible marked permutations are related by $\sim $.
As a consequence, we recover the following corollary, which was already obtained in \cref{cor:factsthm} in a more general context: 

\begin{cor}\label{lm:jfacwdef}
Consider $\alpha^* $ a marked permutation, together with factorizations $\alpha^* = \xi_1^* \star \dots \star \xi_k^* = \rho_1^* \star \dots \star \rho_j^* $ into irreducible marked permutations.
Then $k= j$ and $ \{\xi_1^* ,   \dots ,  \xi_k^* \} = \{ \rho_1^* , \dots , \rho_j^* \}$ as multisets. 
\end{cor}

As we did in \cref{sec:comutfree}, we define $j(\alpha^* ) $  to be the number of irreducible factors in any factorization of $\alpha^*$ into irreducible factors, and define $\mathfrak{fac}(\alpha^*) $ as the multiset of irreducible factors of $\alpha^*$ in $\mathcal{G}(\mathtt{MPer})$.
These are well defined as a consequence of \cref{lm:jfacwdef}.

\begin{defin}[Stability conditions]\label{defin:stab}
A factorization of a marked permutation $\alpha^* $ into irreducible marked permutations,
$$ \alpha^* = \xi_1^* \star \dots \star \xi_j^* \, , $$
or the corresponding word $(\xi_1^* , \dots , \xi_j^*)$ in $\mathcal{W}(\Omega)$, is 

\begin{itemize}
\item \textbf{$i$-$\oplus$-stable} if there are no $\pi, \tau$ $\oplus$-indecomposable permutations such that
$$ \xi_i^* = \bar{1}\oplus \pi \, \text{ and } \xi_{i+1}^* = \tau \oplus \bar{1} \, ;$$

\item \textbf{$i$-$\ominus$-stable} if there are no $\pi, \tau$ $\ominus$-indecomposable permutations such that
$$  \xi_i^* = \tau \ominus \bar{1}  \text{ and } \xi_{i+1}^* = \bar{1}\ominus \pi \, .$$
\end{itemize} 

Such a factorization or word is $\oplus$\textbf{-stable} (resp. $\ominus$\textbf{-stable}) if it is $i$-$\oplus$-stable (resp. $i$-$\ominus$-stable ) for any $i = 1, \dots , j-1$.
Finally, such a factorization or word is \textbf{stable} if it is both $\oplus$-stable and $\ominus$-stable.
\end{defin}

%Similarly, we say that the corresponding word $(s_1^* , \dots , s_j^*)$ in $\mathcal{W}(\Omega)$ is $i$-$\oplus$\textbf{-stable}, $i$-$\ominus$\textbf{-stable}, $\oplus$-stable, $\ominus$-stable and \textbf{stable}

\begin{rem}[Stability reduction]\label{rem:stabred}
If we are given a factorization of $\alpha^*  = \xi_1^* \star \dots \star \xi_j^*$ that is not  $i$-$\oplus$-stable or $i$-$\ominus$-stable for some $i= 1, \dots , j-1$, we can perform an $i$\textit{-reduction}, that maps $(\xi_1^*, \dots ,  \xi_i^*, \xi_{i+1}^*, \dots \star \xi_j^*)$ to $ (\xi_1^*, \dots ,  \xi_{i+1}^*, \xi_i^*, \dots \star \xi_j^*)$.

It is immediate to see that this procedure of finding some $i$ and applying an $i$-reduction always terminates.
The final word is stable and is independent of the order in which we apply the reductions.
\end{rem}

Because of the above, any equivalence class in $\mathcal{W}(\Omega ) _{\sim} $ admits a unique stable word, and a consequence of \cref{thm:isomonoids} is the following.

\begin{cor}[Unique stable factorization]\label{cor:simpleUFT}
Let $\alpha^* $ be a marked permutation.
Then, $\alpha^*$ has a unique stable factorization into irreducible marked permutations.
We refer to it as \textit{the stable factorization of $\alpha^*$}.
\end{cor}

\subsection{Lyndon factorization on marked permutation\label{sec:LFT}}

We introduce an order on permutations, two orders on marked permutations and an order in $\mathcal{W}(\Omega)$.

\begin{defin}[Orders on marked permutations]\label{defin:orders}

The \textit{lexicographic order on permutations} is the lexicographic order when reading the one-line notation of permutations, and is written $\pi \leq_{per} \tau$.

Recall that, for a marked permutation $\alpha^* = (\leq_P, \leq_V)$ in $I$, we define its rank $\mathfrak{rk}(\alpha^* )$ as the rank of $*$ in $I \sqcup \{* \}$ with respect to the order $\leq_P$.
We also write $\alpha $ for referring to the corresponding permutation in the set $I\sqcup \{*\}$.
We define the \textit{lexicographic order on marked permutations}, also denoted $\leq_{per} $, as follows: we say that $\pi^* \leq_{per} \tau^* $ if $\pi <_{per} \tau $ or if $\pi = \tau $ and $\mathfrak{rk} (\pi^*)\leq \mathfrak{rk}(\tau^*)$.

This in particular endows our alphabet $\Omega $ of irreducible marked permutations with an order.
When we compare words on $\mathcal{W}(\Omega )$ we order them lexicographically according to $\leq_{per}$, and denote it simply as $\leq $.

We define the \textit{factorization order} $\leq_{fac}$ on marked permutations as follows:
Let $\pi^*= \xi_1^* \star \dots \star \xi_k^* $ and $\tau^* = \tau_1^* \star \dots \star \tau_j^* $ be the respective unique stable factorizations of $\pi^*$ and $\tau^*$.
Then, we say that $\pi^* \leq_{fac} \tau^* $ if $(\xi_1^*, \dots, \xi_k^*)\leq (\tau_1^*, \dots, \tau_j^*)$ in $\mathcal{W}(\Omega )$.
\end{defin}

\begin{smpl}
On permutations we have $12345 \leq_{per} 132 \leq_{per} 231 \leq_{per} 4123$.
Observe that the empty permutation is the smallest permutation.

On marked permutations, we have $1\bar{3}2 \leq_{per} 13\bar{2} \leq_{per} \bar{2}31 \leq_{per} 412\bar{3} $.
Observe that the trivial marked permutation $\bar{1} $ is the smallest marked permutation.

On words, we have the following examples: $(24\bar{1}3, 31\bar{4}2) \leq (31\bar{4}2, 24\bar{1}3) $, $ (\bar{1}32 , 21\bar{3} ) \leq (\bar{1}432) $ and  $(2\bar{4}13, \bar{1}423, 2\bar{4}13) \leq (2\bar{4}13, 2\bar{4}13, \bar{1}423)$.

On marked permutations, because $(24\bar{1}3, 31\bar{4}2)$ and $(31\bar{4}2, 24\bar{1}3) $ are stable, we have that $ 24\bar{1}3 \star 31\bar{4}2 \leq_{fac} 31\bar{4}2 \star 24\bar{1}3 $.
Notice however, that $(\bar{1}32 , 21\bar{3} ) $ is not a stable word, as it does not satisfy the $1$-$\oplus $-stability condition.
Instead, we have $(\bar{1}32 , 21\bar{3} )^* = 21\bar{3} \star \bar{1}32 \geq_{fac} \bar{1}432$.

Sometimes the orders $\leq_{per}$ and $\leq_{fac}$ on marked permutations do not agree, as exemplified in \cref{fig:orderfig}.

\begin{figure}[h]
\centering
\includegraphics[scale=0.60]{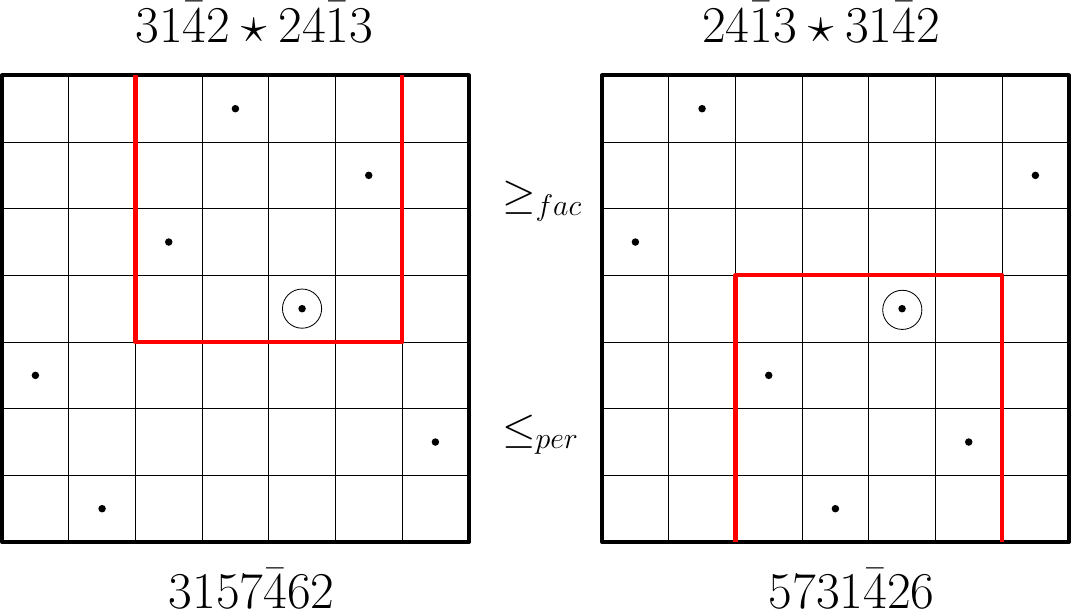}
\caption{Two marked permutations and their order relations.\label{fig:orderfig}}
\end{figure}

\end{smpl}

\begin{rem}\label{rem:goodorder}
If $w_1 \geq w_2$ are stable words in $\mathcal{W}(\Omega) $, then these correspond precisely to the stable irreducible factorizations of $w_1^*, w_2^*$.
Thus, $w_1^* \geq_{fac} w_2^*$.
\end{rem}

\begin{rem}\label{rem:smallestfact}
%\label{rem:verygoodorder}
If a word $w = (\rho_1^*, \dots , \rho_k^*)$ in $\mathcal{W}(\Omega )$ is not $i$-$\oplus$-stable or $i$-$\ominus$-stable, then $\rho_i^* <_{per}\rho_{i+1}^*$.
Thus, from \cref{thm:isomonoids,rem:stabred}, for a fixed marked permutation $\alpha^*$, among all words $w$ on irreducible marked permutations such that $w^* = \alpha^* $, the stable factorization is the biggest one in the lexicographical order.
\end{rem}

We start the discussion on the topic of Lyndon words.
This is useful because the unique factorization theorem obtained in \cref{cor:simpleUFT} for marked permutations is not enough to establish the freeness of $\mathcal{A}(\mathtt{MPer})$ and, as in \cite{vargas14}, Lyndon words are the tool that allows us to describe an improved unique factorization theorem in \cref{thm:Lyndonmperfact}.

\begin{defin}[Lyndon words]
Given an alphabet $A$ with a total order, a word $l \in \mathcal{W}(A) $ is a \textit{Lyndon word} if, for any concatenation factorization $l = a_1 \cdot a_2$ into non-empty words, we have that $a_2\geq l $.
\end{defin}

\begin{smpl}[Examples of Lyndon words]
Consider the (usual) alphabet 
$$\Omega = \{\bar{1} <_{per } \bar{1}2 <_{per }  \bar{1}32 <_{per } \dots  <_{per } 23\bar{1} <_{per } 24\bar{1}3 <_{per } \dots  \} \, ,$$
then $(\bar{1}2, \bar{1}32, \bar{1}2 , 24\bar{1}3)$ is a Lyndon word in this alphabet.
Meanwhile, $(\bar{1}, \bar{1})$ is \textbf{not} a Lyndon word.
\end{smpl}

\begin{defin}[Stable Lyndon marked permutations]\label{defin:centralLyndon}
A word on irreducible marked permutations $w = (\xi_1^*, \dots, \xi_j^*) \in \mathcal{W}(\Omega ) $ is called \textit{stable Lyndon}, or SL for short, if it is a Lyndon word and satisfies the \textit{stability} conditions introduced in \cref{defin:stab}.

A marked permutation $\pi^*$ is called \textit{stable Lyndon}, or SL for short, if there exists an SL word $l =(\xi_1^*, \dots, \xi_j^*)$ such that $l^* = \pi^*$.
We write $\mathcal{L}_{SL} $ for the set of SL marked permutations.
Observe that, from \cref{cor:simpleUFT}, if such an SL word exists it is unique.
\end{defin}

We see latter in \cref{thm:triang} that $\mathcal{L}_{SL} $ is precisely the set that indexes a free basis of $\mathcal{A}(\mathtt{MPer})$.
To establish the unique factorization theorem in the context of marked permutations, we first recall a classical fact in Lyndon words.

\begin{thm}[Unique Lyndon factorization theorem, \cite{chen58}]\label{thm:Lyndonfact}
Consider a finite alphabet $A$ with a total order.
Then any word has a unique factorization, in the concatenation product, into Lyndon words $l_1, \dots  ,l_k$ such that $l_1 \geq \dots \geq l_k$ for the lexicographical order in $\mathcal{W}(A)$.
\end{thm}

This theorem is central in establishing the freeness of the shuffle algebra in \cite{radford79}.
So is the unique factorization into Lyndon marked permutations below, in establishing the freeness of $\mathcal{A}(\mathtt{MPer})$.

\begin{thm}[Unique stable Lyndon factorization theorem]\label{thm:Lyndonmperfact}
Let $\alpha^* $ be a marked permutation.
Then there is a unique sequence of SL words on $\Omega$, $l_1, \dots, l_k $ such that $l_i \geq l_{i+1}$ and $ \alpha^* = l_1^* \star \dots \star l_k^* $.
\end{thm}

We refer to such a sequence of words as the \textit{SL factorization} of $\alpha^*$.

\begin{proof}
The existence follows from \cref{cor:simpleUFT} and \cref{thm:Lyndonfact}.
Indeed, for any marked permutation $\alpha^*$, there is a unique stable factorization $\xi_1^*, \dots , \xi_{j(\alpha^*)}^*$, and from the Lyndon factorization theorem, the word $(\xi_1^*, \dots , \xi_{j(\alpha^*)}^*)\in \mathcal{W}(\Omega)$ admits a factorization into Lyndon words $l_1, \dots , l_k $ such that $l_i \geq l_{i+1}$.
These words are stable because $(\xi_1^*, \dots , \xi_{j(\alpha^*)}^*)$ is stable, and from $(\xi_1^*, \dots , \xi_{j(\alpha^*)}^*) = l_1 \cdots l_k$ we have that $\alpha^* = l_1^* \star \dots \star l_k^*$.
We then obtain the desired sequence of SL words $l_1, \dots , l_k$.

For the uniqueness of such a factorization, suppose we have SL words $m_1 \geq \dots \geq m_{k'} $ that form an SL factorization of $\alpha^*$.
We wish to show that this is precisely the sequence $l_1, \dots , l_k$ constructed above.
Write $m_k = (\rho_{k, 1}^*, \dots , \rho_{k, z_k}^*) $ for $v=1, \dots, k'-1$, where $z_v = |m_v|$ and, for readability purposes, consider as well the re-indexing $m_1 \cdots m_{k'} = (\rho_1^*, \dots , \rho_z^*) $.

First observe that from $\alpha^* = m_1^* \star \dots \star m_{k'}^*$ we get that $\rho_1^* \star \dots \star \rho_z^* $ is a factorization of $\alpha^*$ into irreducibles.
Further, because each $m_j$ is stable, the $i$-$\oplus$-stability and $i$-$\ominus$-stability of this factorization is given for any $i  $ that is not of the form $z_1 + \dots + z_v$, for some $v=1, \dots, k'-1$.
On the other hand, it follows from the Lyndon property that $\rho^*_{k, z_k} >_{per} \rho^*_{k, 1}$; further, because  $m_k \geq m_{k+1} $, we have that $ \rho^*_{k, 1}\geq_{per} \rho^*_{k+1, 1} $.
We conclude that $\rho^*_{k, z_k} >_{per} \rho^*_{k+1, 1}$. Comparing with \cref{rem:smallestfact}, we have the $i$-$\oplus$-stability and $i$-$\ominus$-stability condition for any $i  $ that is of the form $z_1 + \dots + z_v$, for some $v=1, \dots, k'-1$.

Thus, $\rho_1^* \star \dots \star \rho_z^* $ is the stable factorization of $\alpha^*$, so that $(\rho_1^*,  \dots  , \rho_z^* ) = (\xi_1^*, \dots , \xi^*_k ) $.
Further, the sequence $m_1\geq \dots \geq m_{k'} $ is the Lyndon factorization of $(\rho_1^*,  \dots  , \rho_z^* )$, so $(m_1, \dots , m_{k'}) = (l_1, \dots , l_k)$ by the uniqueness in \cref{thm:Lyndonfact}.
\end{proof}

We define $k(\alpha^* )$ to be the number of factors in the stable Lyndon factorization of $\alpha^*$.
Note that for any marked permutation $\alpha^* $, $k(\alpha^* ) \leq j(\alpha^* )$ where we recall that $j(\alpha^* ) $ is the number of irreducible factors in a factorization of $\alpha^*$ into irreducible marked permutations.

\begin{defin}[Word shuffle]
Consider $\Omega$ an alphabet, and let $w, l_1, \dots , l_k\in \mathcal{W}(\Omega)$.
We say that $w=(w_1, \dots , w_j)$ is a \textit{word shuffle} of $ l_1, \dots , l_k$ if $[j]$ can be partitioned into $k$ many disjoint sets $\{q_1^{(i)} <  \dots < q^{(i)}_{|l_i|} \}$, where $i$ runs over $i=1, \dots , k$, such that 
$$ l_i = (w_{q_1^{(i)}}, \dots , w_{q_{|l_i|}^{(i)}} ) \, , $$
for all $i=1 , \dots , k$.
\end{defin}

The following theorem is proven in \cite[Theorem 2.2.2]{radford79} and is the main property of Lyndon factorizations on words that we wish to use here.
This is also the main ingredient in showing that the shuffle algebra is free.

\begin{thm}[Lyndon shuffles - \cite{radford79}]\label{thm:lyndonshuf}
Take $\Omega$ an ordered alphabet, and $l_1\geq \dots \geq l_k$ Lyndon words on $\mathcal{W}(\Omega)$.
Consider $w\in \mathcal{W}(\Omega )$, and assume that $w$ is a word shuffle of $l_1, \dots , l_k$.
Then $w \leq l_1 \cdots l_k$.
\end{thm}

We remark that there are substantial notation differences: particularly, in \cite{radford79} a \textit{prime factorization} is what we call here the \textit{Lyndon factorization}.
Our notation follows \cite{grinberg14}.

\subsection{Freeness of the pattern algebra in marked permutations\label{sec:MTheorems}}

In this section we state the main steps of the proof of \cref{thm:freemperm}.
The proof of the technical lemmas is postponed to \cref{sec:proofoflemmas}.

We consider the set of SL marked permutations $\mathcal{L}_{SL}$, which play the role of free generators, and consider a multiset of SL marked permutations $\{\iota_1^* \geq_{fac} \dots \geq_{fac} \iota_k^* \}$ and the marked permutation $\alpha^* = \iota_1^* \star \dots \star \iota_k^* $.

Then, all the terms that occur on the right hand side of
\begin{equation}\label{eq:prodform}
\prod_{i=1}^{k} \pat_{\iota_i^*} = \sum_{\beta^* }\binom{\beta^*}{\iota_1^*,  \dots , \iota_k^*}\pat_{\beta^* } \, ,
\end{equation}
correspond to \textit{quasi-shuffles} of $\iota_1^*,  \dots , \iota_k^* $.
Below we establish that the marked permutation $\alpha^* = \iota_1^* \star \dots \star \iota_k^*$ is the biggest such marked permutation occurring on the right hand side of \eqref{eq:prodform}, with respect to a suitable total order related to $\leq_{fac}$.
%Notice that $k(\alpha^* ) = k$.

\begin{thm}\label{thm:triang}
Let $\alpha^* $ be a marked permutation, and suppose that $\iota_1^*, \dots , \iota_k^* $ is its Lyndon factorization.
Then there are coefficients $c_{\beta^* } \geq 0$ such that
\begin{equation}\label{eq:triang}
\prod_{i=1}^k \pat_{\iota_i^*} = \sum_{|\beta^*| < |\alpha^* | } c_{\beta^* } \pat_{\beta^* } + \sum_{\substack{|\beta^*| = |\alpha^*| \\ j(\beta^*) < j(\alpha^* ) }} c_{\beta^* } \pat_{\beta^* } + \sum_{\substack{|\beta^*| = |\alpha^*| \\ j(\beta^*) = j(\alpha^* ) \\ \beta^* \leq_{fac} \alpha^* }} c_{\beta^* } \pat_{\beta^* } \, .
\end{equation}
Furthermore, $c_{\alpha^* } \geq 1$.
\end{thm}

With this, the linear independence of all products of the form $\prod_{i=1}^{k} \pat_{\iota_i^*}$ follows from the linear independence of $\{\pat_{\alpha^* } | \alpha^* \in \mathcal{G}(\mathtt{MPer}) \}$, established earlier in \cref{rem:lipatfunc}.
In other terms, \cref{thm:triang} implies \cref{thm:freemperm}.
The technical lemmas necessary to prove \cref{thm:triang} are the following:

\begin{lm}\label{lm:sizeqs}
Let $\beta^*$ be a quasi-shuffle of the marked permutations $\iota_1^*, \dots , \iota_k^*$.
Then, $|\beta^* | \leq |\alpha^*|$.
Furthermore, if $|\beta^* | = |\alpha^*|$, then $j(\beta^*) \leq j(\alpha^*)$.
\end{lm}
%\todo[inline]{I may be using the assumption that each $\iota$ is SL marked permutation below}

\begin{lm}[Factor breaking lemma]\label{lm:bbl}
Let $\beta^*$ be a quasi-shuffle of the marked permutations $\iota_1^*,  \dots , \iota_k^*$, such that $ |\beta^*| = |\alpha^*|$ and $ j(\beta^*) = j(\alpha^* ) $.
Then $\mathfrak{fac} (\beta^* )  = \mathfrak{fac} (\alpha^* ) $.

Furthermore, if for each $i=1, \dots , k$,  
$$ \iota_i^*= \zeta_{1, i}^* \star \dots \star \zeta_{j(\iota_i^*), i}^*$$
is the stable factorization of $\iota_i^*$, then there is a marked permutation $\gamma^*$ with a factorization into irreducibles given by $\gamma^* = \tau_1^* \star \dots \star \tau_{j(\gamma^*) }^*$ such that

\begin{itemize}
\item we have $\mathfrak{fac}(\gamma^* ) = \mathfrak{fac}(\beta^* )$. In particular, $ |\gamma^*| = |\beta^*| = |\alpha^*|$ and $j(\gamma^*) = j(\beta^*) = j(\alpha^* ) $;

\item we have $\gamma^* \geq_{fac} \beta^*$;

\item the word $(\tau_1^*, \dots ,\tau_{j(\gamma^*) }^*)$ is a word shuffle of the words $\{(\zeta_{1, i}^* , \dots , \zeta_{j(\iota_i^*), i}^*)\}_{i=1, \dots , k}$.
\end{itemize}

\end{lm}

These lemmas will be proven below, in \cref{sec:proofoflemmas}.
We remark that, here, the chosen ordering $\leq_{per}$ on the irreducible marked permutations plays a role.
That is, in proving that $\gamma^* \geq_{fac} \beta^*$ we use properties of the order $\leq_{per}$ introduced above, like \cref{rem:smallestfact}.
This is unlike the work in \cite{vargas14}, where any order in the $\oplus$-indecomposable permutations gives rise to a set of free generators.

\begin{lm}[Factor preserving lemma]\label{lm:bplol}
Let $\gamma^*$ be a marked permutation with a factorization into irreducibles given by $\gamma^* = \tau_1^* \star \dots \star \tau_{j(\gamma^* )}^*$, and let $l_1, \dots , l_k$ be SL words, such that
\begin{itemize}
\item  $l_i \geq l_{i+1}$ for each $i=1, \dots , k-1$ and;

\item  The word $(\tau_1^*, \dots ,\tau_{j(\gamma^*) }^*)$ is a word shuffle of the words $\{ l_i\}_{i=1, \dots , k}$.
\end{itemize}

Then, $\gamma^* \leq_{fac} l_1^* \star  \dots \star l_k^*$.
\end{lm}

In the remaining of the section we assume these lemmas and conclude the proof of the freeness of the pattern algebra, by establishing \cref{thm:triang}.

\begin{proof}[Proof of \cref{thm:triang}]\label{prf:triang}
Because a product is a quasi-shuffle of its factors, it follows from \cref{rem:inflqshuf} and from $\alpha^* = \iota^*_1 \star \dots \star \iota^*_k$ that $c_{\alpha^* } \geq 1$.
We will complete the proof if we show that any $\beta^* $ that is a \textit{quasi-shuffle} of $\iota_1^* \dots , \iota_k^* $ satisfies either
\begin{itemize}
\item $|\beta^*| < |\alpha^* |$;

\item or $|\beta^*| = | \alpha^*|$ and $j(\beta^*) < j(\alpha^*) $;

\item or $|\beta^*| = |\alpha^* |$, $j(\beta^*) = j(\alpha^*) $ and $\beta^* \leq_{fac} \alpha^* $.
\end{itemize}

Suppose that $\beta^* $ is a quasi-shuffle of $\iota_1^* , \dots , \iota_k^*$.
For each $i=1, \dots k$, let $l_i$ be the SL word corresponding to the SL marked permutation $\iota_i^*$.
From \cref{lm:sizeqs} we only need to consider the case where $|\beta^*| = |\alpha^* |$ and $j(\beta^*) = j(\alpha^*) $.

From \cref{lm:bbl} we have that in this case, $\mathfrak{fac}(\beta^*) = \mathfrak{fac}(\alpha^*)$, and there is a marked permutation $\gamma^* $ that satisfies $\mathfrak{fac}(\gamma^*) = \mathfrak{fac}(\beta^*)$, $\gamma^* \geq_{fac} \beta^*$, and also has a factorization into irreducibles that is a word shuffle of $l_1 , \dots , l_k$.
Thus, from \cref{lm:bplol} we have that $\gamma^* \leq_{fac} l_1^*  \star \dots \star l_k^* = \alpha^*$.
It follows that $\beta^* \leq_{fac} \gamma^* \leq_{fac} \alpha^*$, as desired.
\end{proof}

\subsection{Proof of unique factorization theorem\label{sec:UFTProof}}

We start this section with the concept of \textit{DC intervals} and \textit{DC chains} of a marked permutation $\alpha^*$.
We will see that these are closely related to factorizations of $\alpha^*$, and we will exploit this correspondence to prove \cref{thm:isomonoids}.

Recall that if $\beta^* $ is a marked permutation in $X$, and $I\subseteq X$, we denote $I^* \coloneqq I \sqcup \{ * \}$ for simplicity of notation.
We also write $\mathbb{X}(\beta^* ) = X$.

\begin{defin}
Let $\beta^* $ be a marked permutation on the set $X$, \textit{i.e.} $\beta^* = (\leq_P, \leq_V)$ is a pair of orders on the set $X^*$.
A \textit{doubly connected interval} on $\beta^* $, or a \textit{DC interval} for short, is a set $I \subseteq X $ such that $I^* $ is an interval on both orders $\leq_P, \leq_V$.
A \textit{proper} DC interval is a DC interval $I$ such that $I \neq X$.
\end{defin}

Note that $X$ and $\emptyset $ are always DC intervals.
Note as well that if $I_1, I_2$ are DC intervals, then $I_1\cup I_2$ is also a DC interval.

\begin{rem}\label{rem:DCintervalsarebig}
Suppose that $\beta^* = (\leq_P, \leq_V ) $ is a marked permutation, and $I $ is a DC interval of $\beta^*$.
Consider $m_P, M_P\in X^*$ the minimal, respectively maximal, element for the order $\leq_P$.
If $m_P, M_P \in I^*$, then $I = X$.

Symmetrically, if $m_V, M_V\in X^* $ are the minimal, respectively maximal, elements for the order $\leq_V$, and $m_V, M_V \in I^*$, then $I = X$.
\end{rem}

For a marked permutation $\beta^*$, we use the notation $m_P, M_P, m_V, M_V $ in the remaining of the section to refer to its respective extremes, as in \cref{rem:DCintervalsarebig}.

\begin{defin}
For a DC interval $I$ of a marked permutation $\alpha^* =(\leq_P, \leq_V)$, we define the permutation $\alpha^* \setminus_{I^*} $ in the set $ I^c$ resulting from the restriction of the orders $\leq_P, \leq_V$ to the set $I^c $.
Alternatively, this is the permutation resulting from the removal of the marked element in $\alpha^*|_{I^c}$.
\end{defin}

\begin{rem}\label{rem:simpleDCinterv}
If $\alpha^* = a_1^* \star a_2^* $, then $\mathbb{X}(a_2^*) $ is a DC interval in $\alpha^* $.
On the other hand, if $I$ is a DC interval of $\alpha^*$, then $\alpha^* = \alpha^*|_{I^c} \star \alpha^*|_{I}$, so right factors of $\alpha^* $ are in bijection with DC intervals of $\alpha^*$.
Furthermore, $I$ is a maximal proper DC interval of $\alpha^* $ if and only if $\alpha^*|_{I^c}$ is irreducible.

To a factorization $\beta^* = \rho_1^* \star \dots \star \rho_j^*$ of a marked permutation $\beta^* \in \mathtt{MPer}[X]$ it corresponds a chain of DC intervals 
$$\emptyset = J_{j+1} \subsetneq J_j \subsetneq \dots \subsetneq J_1 = X \, .$$

The chain is defined by $J_k = \mathbb{X} (\rho_k^* \star \dots \star \rho_j^*)$ and satisfies $\beta^*|_{J_k \setminus J_{k+1} } = \rho_k^* $  for any $k = 1, \dots , j$.
This chain of DC intervals is maximal if and only if the original factorization has only irreducible factors.
\end{rem}

\begin{rem}\label{rem:oplusDCinterv}
The marked permutation $\alpha^* $ is $\oplus$-decomposable if and only if there is a DC interval $I$ of $\alpha^* $ such that $I^*$ contains both $m_P, m_V$, or contains both $M_P, M_V$.

In the first case, $\alpha^*$ factors as $\beta^*_1 \oplus \beta_2$, where $\beta_1^* = \alpha^*|_I $ and $\beta_2 = \alpha^*\setminus_{I^*}$.
In the second case, $\alpha^*$ factors as $\beta_1 \oplus \beta_2^*$, where $\beta_1 = \alpha^*\setminus_{I^*} $ and $\beta_2^* = \alpha^*|_{I}$.

Similarly, $\alpha^* $ is $\ominus$-decomposable if and only if there is a DC interval $I$ of $\alpha^* $ such that $I^*$ contains both $M_P, m_V$, or contains both $m_P, M_V$.

In the first case, $\alpha^*$ factors as $\beta^*_1 \ominus \beta_2$, where $\beta_1^* = \alpha^*|_I $ and $\beta_2 = \alpha^*\setminus_{I^*}$.
In the second case, $\alpha^*$ factors as $\beta_1 \ominus \beta_2^*$, where $\beta_1 = \alpha^*\setminus_{I^*} $ and $\beta_2^* = \alpha^*|_I$.
\end{rem}

This characterization of $\oplus$-decomposable marked permutations will be useful in the proof of the classification of the factorizations below, specifically in \cref{lm:simplecont}.
It is also useful to characterize all $\oplus$-decomposable marked permutations that are irreducible in \cref{sec:PrimEl}.

%\begin{obs}[$\oplus$-factorization and $\ominus$-factorization in permutations]\label{lm:ofactsinperm}
%Let $\alpha $ be a permutation.
%Then there are unique $\oplus$-indecomposable permutations $a_1, \dots , a_k$ and $\ominus$-indecomposable permutations $b_1, \dots, b_j$ such that 
%$$ \alpha = a_1 \oplus \dots \oplus a_k  = b_1 \ominus \dots \ominus b_j \, .$$
%\end{obs}
%
%The following is a direct consequence of \cref{lm:ofactsinperm}.

\begin{obs}[$\oplus$-factorization in marked permutations]\label{obs:oplusfact}
Let $\alpha^*$ be a marked permutation.
Then there are unique $\oplus$-indecomposable permutations $\epsilon_1, \dots , \epsilon_k$,  $\lambda_1, \dots, \lambda_j$ and $\beta^* $ an $\oplus$-indecomposable marked permutation such that 
$$ \alpha^* = \epsilon_1 \oplus \dots \oplus \epsilon_k \oplus \beta^* \oplus \lambda_j \oplus \dots \oplus \lambda_1 \, .$$

In this case, we say that $\alpha^* $ is $q$-$\oplus$-decomposable, where $q=k+j$.
\end{obs}

%To $\beta^*$ we call the $\oplus $-kernel of $\alpha^* $, and write $\ker^{\oplus}(\alpha^* )  = \beta^*$.

%
%Further, there are unique $\ominus$-indecomposable permutations $c_1, \dots , c_{k'}$ and $  d_1, \dots, d_{j'}$, and a unique $\ominus$-indecomposable marked permutation $\gamma^*$ such that 
%$$ \alpha^* = c_1 \ominus \dots \ominus c_{k'} \ominus \gamma^* \ominus d_{j'} \ominus \dots \ominus d_1 \, .$$
%
%In this case, we say that $\alpha^* $ is $(k'+j')$-$\ominus$-decomposable.
%To $\gamma^*$ we call the $\ominus$-kernel of $\alpha^* $, and write $\ker^{\ominus}(\alpha^* )  = \gamma^*$.

\begin{smpl}[$q$-$\oplus$-decomposition]
Consider the marked permutation $21\bar{3}54$.
This is a $2$-decomposable marked permutation, as it admits the $\oplus$-factorization
$$ 21\bar{3}54 = 21 \oplus \bar{1} \oplus 21 \, . $$
\end{smpl}

A marked permutation is $\oplus $-indecomposable if and only if it is 0-$\oplus$-decomposable.
%Similarly, a marked permutation is $\ominus $-indecomposable if and only if it is 0-$\ominus$-decomposable.

\begin{lm}[Irreducible $\oplus$-factor lemma]\label{lm:opcontinuity}
Suppose that $\alpha^* $ is a marked permutation that has the following $\oplus $-factorization
$$ \alpha^* =  \epsilon_1 \oplus \dots \oplus \epsilon_u \oplus \beta^* \oplus \lambda_v \oplus \dots \oplus \lambda_1 \, , $$
where $u+v >  0$.
Consider any factorization $\alpha^* = \sigma^* \star \gamma^* $, where $\sigma^* $ is an irreducible marked permutation.

\begin{itemize}
\item If both $u>0$ and $v>0$, then either $\sigma^* = \bar{1}\oplus \lambda_1 $ or $\sigma^* = \epsilon_1 \oplus \bar{1}$.

\item If $u=0$, then $\sigma^* = \bar{1}\oplus \lambda_1 $.

\item If $v=0 $, then $\sigma^* = \epsilon_1 \oplus \bar{1}$.
\end{itemize}

\end{lm}

\begin{proof}
Let us deal with the case where both $u > 0 $ and $v > 0$ first.
Assume that neither  $\sigma^* = \bar{1}\oplus \lambda_1 $ nor $\sigma^* = \epsilon_1 \oplus \bar{1}$.
Consider the following DC intervals, which are all distinct
$$ Y_1 = \mathbb{X}(\epsilon_2\oplus \dots \oplus \epsilon_u \oplus \beta^* \oplus \lambda_v \oplus \dots \oplus \lambda_1 )  \, , $$
$$ Y_2 = \mathbb{X}(\epsilon_1\oplus \dots \oplus \epsilon_u \oplus \beta^* \oplus \lambda_v \oplus \dots \oplus \lambda_2 ) \, , $$
$$ Y = \mathbb{X}(\gamma^* )\, . $$

Note that $m_P \in Y_2^*\setminus Y_1^*$ and $M_P\in Y_1^*\setminus Y_2^*$ by construction.
Note also that $\alpha^*|_{Y_1^c} = \epsilon_1\oplus \bar{1} $  and $ \alpha^*|_{Y_2^c} = \bar{1} \oplus  \lambda_1$ are irreducible marked permutations, so both $Y_1, Y_2$ are maximal proper DC intervals.
Furthermore, $Y$ is also a maximal proper DC interval.
This maximality gives us that $Y \cup Y_1  = X = Y\cup Y_2 $.

From $Y \cup Y_1  = X$ and $m_P\not\in Y_1^* $, we have that $m_P \in Y^*$, and from $Y\cup Y_2 = X $ and $M_P\not\in Y_2^* $ we get that $M_P \in Y^*$.
From \cref{rem:DCintervalsarebig}, this is a contradiction with the fact that $Y$ is a proper DC interval, thus contradicting the assumption that neither $ \sigma^* = \bar{1}\oplus \lambda_1 $ nor $ \sigma^* = \epsilon_1 \oplus \bar{1} $, as desired.

\medskip

Now suppose that $u=0$, and $v > 0 $.
Assume for the sake of contradiction that $\sigma^* \neq \bar{1}\oplus \lambda_1$, and define the following distinct DC intervals
$$ Y_2 = \mathbb{X}(\beta^* \oplus \lambda_v \oplus \dots \oplus \lambda_2 ) \text{ and } Y = \mathbb{X}(\gamma^* )\, . $$
Because both $\sigma^* $ and $ \bar{1}\oplus \lambda_1 $ are irreducible, the DC intervals $Y_2, Y $ are both maximal proper, so $X = Y \cup Y_2$.
Further, $m_P \in Y_2^* $ by construction, so $M_P\in Y^* $ by \cref{rem:DCintervalsarebig}.
Consider the DC interval $I = \mathbb{X}(\beta^* ) $, and notice that $m_P \in I^* $ by construction.
Thus we have that $I^* \cup Y^*  = X^* $, and $Y^c \subseteq I$.

\begin{figure}[h]
\centering
\includegraphics[scale=1]{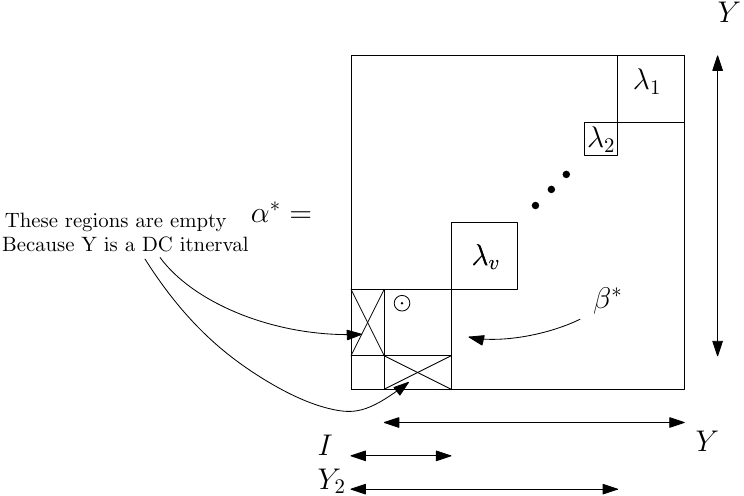}
\caption{A decomposition of $\beta^*$, whenever the DC interval $Y$ is proper\label{fig:splitbeta}}
\end{figure}

We have the following decomposition, depicted in \cref{fig:splitbeta}:
$$\beta^* = \alpha^*\setminus_{Y^*} \oplus \alpha^*|_{I\cap Y} \, . $$
Notice that $ \alpha^*\setminus_{Y^*}$ is not the empty permutation, because $ Y $ is a proper DC interval.
This is a contradiction with the fact that $\beta^*$ is $\oplus$-indecomposable, so $\sigma^* = \bar{1}\oplus \lambda_1$.

The case where $u> 0 $ and $v=0$ is done in a similar way, so the result follows.
\end{proof}

\begin{cor}\label{cor:opcontinuity}
Consider a marked permutation $\alpha^*$ that is $q$-$\oplus$-decomposable, so that $ \alpha^* =  \epsilon_1\oplus \dots \oplus \epsilon_u \oplus \beta^* \oplus \lambda_v \oplus \dots \oplus \lambda_1 $ is the $\oplus$-decomposition of $\alpha^* $. 
If $\sigma_1^* \star \dots \star \sigma_j^*$ is a factorization of $\alpha^*$ into irreducibles, then $j\geq q$ and

\begin{enumerate}
\item $\sigma_{q+1}^*\star \dots \star \sigma_j^* = \beta^* $.

\item By applying $\oplus$-relations to $(\sigma_1^*, \dots , \sigma_q^*)$ we can obtain
$$ (\bar{1}\oplus \lambda_1, \dots, \bar{1}\oplus \lambda_v, \epsilon_1\oplus \bar{1}, \dots , \epsilon_u\oplus \bar{1} )\, . $$
\end{enumerate}
\end{cor}

\begin{proof}
We use induction on $q$.
The base case is $q=0$, where there is nothing to establish in 2. and we need only to show that 
$$ \sigma_{1}^* \star \dots \star \sigma_j^* = \beta^* \, ,$$
which follows because $ \sigma_{1}^* \star \dots \star \sigma_j^* = \alpha^* = \beta^* $.

Now for the induction step, we assume that $q\geq 1$ .
From \cref{lm:opcontinuity}, $\sigma_1^* $ is either $\epsilon_1 \oplus \bar{1} $ or $\bar{1} \oplus \lambda_1$, thus according to \cref{rem:integerdomain} $\zeta^* := \sigma_2^* \star \dots \star \sigma_k^* $ is either
$$ \epsilon_2\oplus \dots \oplus \epsilon_u \oplus \beta^* \oplus \lambda_v \oplus \dots \oplus \lambda_1\, , $$
or
$$ \epsilon_1\oplus \dots \oplus \epsilon_u \oplus \beta^* \oplus \lambda_v \oplus \dots \oplus \lambda_2. $$

Without loss of generality, assume the first case.
Then $\zeta^* $ is $(q-1)$-$\oplus$-decomposable, so the induction hypothesis applies to the factorization of $\zeta^* $.
That is, we have that 
$$\sigma_{q+1}^* \star \dots \star \sigma_j^* =  \beta^* \,  .$$

Furthermore, the induction case gives us that $ (\bar{1}\oplus \lambda_1, \dots, \bar{1}\oplus \lambda_v, \epsilon_2\oplus \bar{1}, \dots , \epsilon_u\oplus \bar{1} ) $ can be obtained from $(\sigma_2^*, \dots ,  \sigma_q^* )$ by a series of $\oplus $-relations.

So, by only using the $\oplus$-relations and the inductive hypothesis we deduce that:
\begin{equation}
\begin{aligned}
 (&\bar{1}\oplus \lambda_1, \dots, \bar{1}\oplus \lambda_v, \epsilon_1\oplus \bar{1} , \epsilon_2\oplus \bar{1}, \dots , \epsilon_u\oplus \bar{1} ) \\
\sim \, \, \,  \quad \quad (& \epsilon_1\oplus \bar{1}, \bar{1}\oplus \lambda_1, \dots, \bar{1}\oplus \lambda_v, \epsilon_2\oplus \bar{1}, \dots , \epsilon_u\oplus \bar{1} ) \\
\sim \, \, \,  \quad \quad (&\sigma_1^*, \sigma_2^*, \dots ,  \sigma_q^* ) 
\end{aligned}
\end{equation}

concluding the the proof.
\end{proof}

We remark that \cref{obs:oplusfact,lm:opcontinuity,cor:opcontinuity} have counterparts for the $\ominus$ operation.

%\begin{lm}[Irreducible $\ominus$-factor lemma]\label{lm:omcontinuity}
%Suppose that $\alpha^* $ is a marked permutation that has the following $\ominus $-factorization
%$$ \alpha^* =  a_1\ominus \dots \ominus a_u \ominus \beta^* \ominus b_v \ominus \dots \ominus b_1 \, , $$
%where $u+v >  0$.
%Consider any factorization $\alpha^* = s^* \star \gamma^* $, where $s^* $ is an irreducible marked permutation.
%
%\begin{itemize}
%\item If both $u>0$ and $v>0$, then either $s^* = \bar{1}\ominus b_1 $ or $s^* = a_1 \ominus \bar{1}$.
%
%\item If $u=0$, then $s^* = \bar{1}\ominus b_1 $.
%
%\item If $v=0 $, then $s^* = a_1 \ominus \bar{1}$.
%\end{itemize}
%\end{lm}
%
%\begin{proof}
%The proof of this result follows the lines of \cref{lm:opcontinuity}
%\end{proof}
%
%\begin{cor}\label{cor:omcontinuity}
%Consider a marked permutation $\alpha^*$ that is $n$-$\ominus$-decomposable, so that $ \alpha^* =  a_1\ominus \dots \ominus a_u \ominus \beta^* \ominus b_v \ominus \dots \ominus b_1 $ is the $\ominus$-decomposition of $\alpha^* $. 
%If $s_1^* \star \dots \star s_j^*$ is a factorization of $\alpha^*$, then
%
%\begin{enumerate}
%\item $s_{n+1}^*\star \dots \star s_j^* = \ker^{\ominus } (\alpha^* )$.
%
%\item By applying $\ominus$-relations to $(s_1^*, \dots , s_n^*)$ we can obtain
%$$ (\bar{1}\ominus b_1, \dots, \bar{1}\ominus b_v, a_1\ominus \bar{1}, \dots , a_u\ominus \bar{1} )\, . $$
%\end{enumerate}
%\end{cor}
%
%\begin{proof}
%This proof is similar to the proof of \cref{cor:opcontinuity}.
%\end{proof}

\begin{lm}[Indecomposable irreducible factor lemma]\label{lm:simplecont}
Suppose that $\alpha^* $ is a non-trivial marked permutation that is both $\oplus$-indecomposable and $\ominus$-indecomposable.
Consider two factorizations of $\alpha^* $,  $ \sigma_1^* \star \beta^*_1 $ and $ \sigma_2^*  \star \beta_2^* $, where $\sigma_1^* , \sigma_2^* $ are irreducible marked permutations.
Then $\sigma_1^* = \sigma_2^* $ and $\beta_1^* = \beta_2^*$.
\end{lm}

\begin{proof}
According to \cref{rem:simpleDCinterv}, it suffices to see that $\alpha^* $ only has one maximal proper DC interval.
Suppose otherwise, that there are $Y_1, Y_2$ distinct proper maximal DC intervals, for the sake of contradiction.
Further, consider $m_P, M_P, m_V, M_V$ the usual maximal and minimal elements in $ \mathbb{X}(\alpha^* )^* $.

We have that $Y_1 \cup Y_2 = X$, so $m_P, M_P, m_V, M_V \in Y_1^* \cup Y_2^* $.
Without loss of generality suppose that $m_P\in Y_1^*$.
We know that $\{m_V, M_V\} \not\subseteq Y_2^*$, so there are only two cases to consider:

\begin{itemize}
\item We have that $m_P, m_V \in Y_1^* $.
Then $\alpha^* $ is $\oplus$-decomposable, according to \cref{rem:oplusDCinterv}.

%Indeed, if $x\in Y_1^*  , y\not\in Y_1^*  $, the DC interval condition implies that $m_P \leq_P x <_P y$ and $m_V \leq_V x <_V y $, so we have that 
%$$\alpha^* = \alpha^*|_{Y_1} \oplus \alpha^* \setminus_{Y_1^*} \, . $$

\item We have that $m_P, M_V \in Y_1^* $.
Then $\alpha^* $ is $\ominus$-decomposable, according to \cref{rem:oplusDCinterv}.

%Indeed, if $x\in Y_1, y\not\in Y_1 $, the DC interval condition implies that $m_P \leq_P x <_P y$ and $M_V \geq_V x >_V y $, so we have that 
%$$\alpha^* = \underbrace{\alpha^*|_{Y_1}}_{\text{ A marked permutation }} \ominus \underbrace{\alpha^* \setminus_{Y_1^*}}_{\text{ A non-marked permutation}} \, . $$

\end{itemize}

In either case, we reach a contradiction.
It follows that $\beta^*_1 = \beta^*_2 $ and $\sigma_1^* = \sigma^*_2$.
\end{proof}

%We will describe precisely which words in $\mathcal{W}( \Omega)$ are mapped through  $\star : \mathcal{W}(\Omega ) \to \mathcal{G}(\mathtt{MPer}) $  to the same

\begin{proof}[Proof of \cref{thm:isomonoids}]
Consider a marked permutation $\alpha^*$, and take words $w_1 = (\xi_1^*,  \dots ,  \xi_k^*)$ and $w_2 = (\rho_1^* ,  \dots , \rho_j^*)$ in $\mathcal{W}(\Omega)$ such that
\begin{equation}\label{eq:givenrelation}
\xi_1^* \star \dots \star \xi_k^* = \rho_1^* \star \dots \star \rho_j^* =: \alpha^* \, ,
\end{equation}
and assume for the sake of contradiction that these can be chosen with $w_1 \not\sim w_2$.
Further choose such words minimizing $k+j = | w_1| + | w_2|$.
We need only to consider four cases:

{\bf The marked permutation $\alpha^* = \bar{1}$ is trivial:} then by a size argument we have that both words $ w_1 $ and $w_2 $ are empty.
Thus  $ w_1 = w_2 $.

{\bf The marked permutation $\alpha^* $ is indecomposable:}
in this case, from \cref{lm:simplecont} we know that $\xi_1^* = \rho^*_1$ and that $\xi_2^* \star \dots\star \xi_k^* = \rho_2^* \star \dots \star \rho_j^* $.
Thus, by minimality, $(\xi_2^*, \dots , \xi_k^*) \sim ( \rho_2^*,  \dots  ,  \rho_j^*)$, which implies $(\xi_1^*, \dots , \xi_k^*)\sim ( \rho_1^*,  \dots  ,  \rho_j^*) $ and that is a contradiction.

{\bf The marked permutation $\alpha^* $ is $\oplus$-decomposable:}
Assume now that $\alpha^* $ is $q$-$\oplus$-decomposable, for some $q>0$.
That is, there are $\oplus$-indecomposables such that 
$$\alpha^* = \epsilon_1 \oplus \dots \oplus \epsilon_u \oplus \beta^* \oplus \lambda_v \oplus \dots \oplus \lambda_1, \quad u+v = q \,  . $$

Then \cref{cor:opcontinuity} tells us that $q \leq k, j$ and that
\begin{equation}\label{eq:kerdec}
 \xi_{q+1}^* \star \dots \star \xi_k^* = \beta^*  = \rho_{q+1}^* \star \dots \star \rho_j^* \, . 
\end{equation}
Further, we also have that
\begin{equation}\label{eq:oplusdec}
(\xi_1^* , \dots ,  \xi_q^*) \sim (\bar{1}\oplus \lambda_1, \dots, \bar{1}\oplus \lambda_v, \epsilon_1\oplus \bar{1}, \dots , \epsilon_u\oplus \bar{1} ) \sim (\rho_1^* , \dots , \rho_q^* )  \, . 
\end{equation}

Because $q > 0$, from \eqref{eq:kerdec} and the minimality of $(\xi_1^*, \dots , \xi_k^*),  ( \rho_1^*,  \dots  ,  \rho_j^*) $ we have that
 $$(\xi_{q+1}^*, \dots , \xi_k^*)\sim ( \rho_{q+1}^*,  \dots  ,  \rho_j^*) \, .$$
This, together with \eqref{eq:oplusdec} gives us $(\xi_1^*, \dots , \xi_k^*)\sim ( \rho_1^*,  \dots  ,  \rho_j^*) $ and that is a contradiction.

{\bf The marked permutation $\alpha^* $ is $\ominus$-decomposable: }
This case is similar to the previous one.

We conclude that a factorization of a marked permutation into irreducibles is unique up to the relations in \cref{rem:oplusominusrelations}.
\end{proof}

\subsection{Proofs of \cref{lm:bplol,lm:bbl,lm:sizeqs}\label{sec:proofoflemmas}}

We start by fixing some notation.
Consider a multiset $\{\iota_1^* \geq \dots \geq \iota_k^* \}$ of SL marked permutations, and let $\alpha^*$ be the marked permutation with Lyndon factorization $(\iota_1^*, \dots , \iota_k^*)$ which exists and is unique by \cref{thm:Lyndonmperfact}, and has $\alpha^* = \iota_1^* \star \dots  \star \iota_k^*$.

Note that  $\mathcal{S} := \biguplus_{i=1}^k \mathfrak{fac} ( \iota_i^*) = \mathfrak{fac}(\alpha^* )$ as multisets, from \cref{lm:jfacwdef}.

\begin{proof}[Proof of \cref{lm:sizeqs}]
Let $\beta^* $ be a marked permutation in $X$.
By \cref{defin:qsands}, there are sets $I_1, \dots , I_k $ such that:
\begin{equation}\label{eq:cover}
I_1 \cup \dots \cup I_k = X \text{ and } \beta^*|_{I_i} \sim \iota_i^* \, \, \, \forall i  =1, \dots , k \, .
\end{equation}

Suppose that $\beta^* $ has stable factorization $\beta^* = \rho_1^* \star \dots  \star \rho_j^*$, where $j = j(\beta^*)$, with a corresponding maximal DC interval chain
$$\emptyset = J_{j+1} \subsetneq J_j \subsetneq  \dots \subsetneq J_1 = X \, , $$
as given in \cref{rem:simpleDCinterv}, so that $ \beta^*|_{J_p \setminus J_{p+1}}= \rho_p $ for $p=1, \dots , j$.

Since $\beta^* $ is a quasi-shuffle of $\iota_1^*, \dots , \iota_k^*$, the sets $I_1, \dots , I_k $ cover $X$ and so
\begin{equation}\label{eq:sizecomp}
 |\beta^*  | = | X |\leq \sum_{i=1}^k | I_i| = \sum_{i=1}^k |\iota_i^* | = |\alpha^* |\, , 
\end{equation}
so we conclude that $|\beta^* | \leq |\alpha^* |$.

Now, assume that we have $|\beta^* | = |\alpha^* |$, so we have equality in \eqref{eq:sizecomp}, thus the sets $I_i$ are mutually disjoint.
Recall that $j(\alpha^* ) = \sum_i j(\iota_i^* )$, as shown in \cref{lm:jfacwdef}.
Therefore, we only need to establish that $\sum_i j(\iota_i^* ) \geq j$.
For each $i= 1, \dots ,  k $, the following is a DC interval chain of $\iota_i^*$, which is not necessarily a maximal one:
\begin{equation}\label{eq:factoreddcchain}
\emptyset = J_{j+1}\cap I_i \subseteq J_j\cap I_i \subseteq \dots \subseteq J_1\cap I_i = I_i \, , 
\end{equation}
so let us consider the set $U_p:=\{i\in [ k ] | J_{p+1}\cap I_i \neq J_p \cap I_i \}$.

First it is clear that each $U_p$ is non empty, as otherwise we would have 
$$J_{p+1} = \bigcup_i J_{p+1}\cap I_i = \bigcup_i J_p \cap I_i = J_p \, .  $$

On the other hand, the length of the DC chain is given precisely by the number of strict inequalities in \eqref{eq:factoreddcchain}, that is by $| \{p \in [j] | i \in U_p\} |$. 
From \cref{lm:jfacwdef,rem:simpleDCinterv}, this is at most $j(\beta^*|_{I_i} ) = j(\iota_i^*)$, so
\begin{equation}\label{eq:jbound}
 \sum_i j(\iota_i^* ) \geq \sum_i |\{p \in [j] | i \in U_p\}| = \sum_{p=1}^j | U_p | \geq j \, .
\end{equation}

So we conclude that $j(\alpha^*) \geq  j(\beta^*)$.
\end{proof}

\begin{proof}[Proof of \cref{lm:bbl}]
Assume now that $|\beta^* | = |\alpha^* |$ and $j(\beta^*) = j(\alpha^*)$.
As a corollary of the proof above, and using the same notation, we have that the sets $I_i$ are pairwise disjoint, and we obtain equality all through \eqref{eq:jbound}, so that each $U_p$ is in fact a singleton.

We wish to split the indexing set $[j(\gamma^*)]$ into $k$ many disjoint increasing sequences $q_1^{(i)} < \dots < q_{j(\iota_i^*)}^{(i)} $ for $i=1, \dots , k$ such that 
\begin{equation}\label{eq:stableseqsplit}
\iota_i^*= \tau_{q_{1}^{(i)}}^* \star \dots \star \tau_{q_{j(\iota_i^*)}^{(i)}}^*  \, ,
\end{equation}
is precisely the stable factorization of each $\iota_i^*$.

Define the map $ \zeta: [j] \to [k] $ that sends $p $ to the unique element in $U_p$.
This map will give us the desired increasing sequences.
It satisfies
$$ | \zeta^{-1} ( i )| = |\{p \in [j] | i \in U_p\}| = j(\iota_i^* ) \, \text{ for } \, i=1 , \dots , k \, .$$
%
%\begin{figure}[h]
%\centering
%\includegraphics[scale=0.6]{DCchain.pdf}
%\caption{\label{fig:DCchain}A description of the DC chains in $\beta^*$ and $\iota^*_i$.}
%\end{figure}

Write $\zeta^{-1}(i)=\{q^{(i)}_1 < \dots  < q^{(i)}_{j(\iota_i^*) } \}\subseteq [j]$ and set $q^{(i)}_{j(\iota_i^*) + 1 }= j+1$.
In the following, we identify the marked permutations $\iota_i^*$ and $\beta^*|_{I_i}$ in order to find a factorization of $\iota^*$ into irreducibles.
Specifically, we get that

\begin{equation}\label{eq:factrho}
\iota_i^* = \beta^*|_{I_i} = \rho_1^*|_{I_i\cap J_1\setminus J_2} \star \dots \star \rho_j^*|_{I_i\cap J_j\setminus J_{j+1}} 
\end{equation}
is a factorization of $\iota_i^*$.
Because each $U_p$ is a singleton, $I_i\cap J_p\setminus J_{p+1} $ is either $J_p\setminus J_{p+1}$ or $\emptyset$, thus $\rho_p^*|_{I_i\cap J_p\setminus J_{p+1}} $ is either $\rho_p^*$ of $\bar{1}$, and the factorization in \eqref{eq:factrho}, after removing the trivial terms, becomes the following factorization into irreducibles:

\begin{equation}\label{eq:simplefactrho}
\iota_i = \rho_{q^{(i)}_1}^*\star \dots \star \rho_{q^{(i)}_{j(\iota^*_i)}}^*\, .
\end{equation}
Then, we indeed have that $\mathfrak{fac}(\beta^* ) = \biguplus_{i=1}^k \mathfrak{fac}(\iota_i^*)$.
We now construct the desired marked permutation $\gamma^* $.

According to \cref{lm:jfacwdef}, for each $i$, the unique stable factorization of $\iota_i^*$ results from $\iota_i^* = \rho_{q^{(i)}_1}^* \star \dots \star \rho_{q^{(i)}_{j(\iota_i^*)}}^*$ by a permutation of the factors.
Thus, for each $i$ we can find indices $p^{(i)}_1, \dots , p^{(i)}_{j(\iota_i^*)}$ such that $\{p^{(i)}_1, \dots , p^{(i)}_{j(\iota_i^*)}\} = \{ q^{(i)}_1, \dots , q^{(i)}_{j(\iota_i^*)} \}$ and 
$$\iota^*_i = \rho_{q^{(i)}_1}^* \star \dots \star \rho_{q^{(i)}_{j(\iota_i^*)}}^* = \rho_{p^{(i)}_1}^* \star \dots \star \rho_{p^{(i)}_{j(\iota_i^*)}}^* $$
is the stable factorization of $\iota^*_i$.

For each $s\in [j]$, if $ s = q_u^{(i)}$ for some integers $i, u$, define $\tau_s \coloneqq \rho_{p_u^{(i)}}^*$ and finally define
$$\gamma^* = \tau_1^* \star \dots \star \tau_j^* \, . $$

\begin{figure}[h]
\centering
\includegraphics[scale=1]{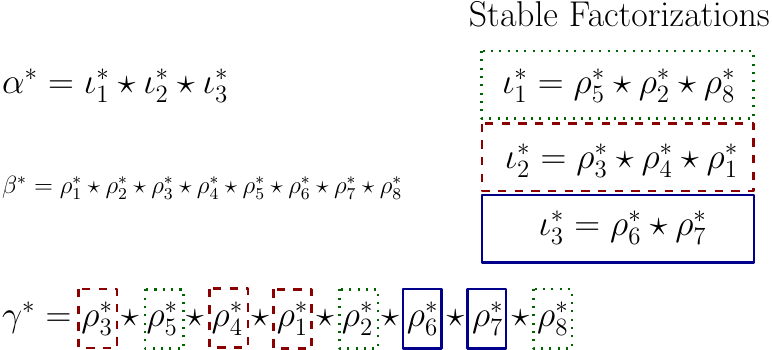}
\caption{\label{fig:gammaconstr} Example of construction of $\gamma^* $ for $j=8$.}
\end{figure}

A graphical description of this construction for $j=8$ is given in \cref{fig:gammaconstr}.
We claim that $\gamma^* $ satisfies the three conditions described in the lemma.
First, it is clear that $\mathfrak{fac}(\beta^* ) = \biguplus_{i=1}^k \mathfrak{fac}(\iota_i^*) = \mathfrak{fac}(\gamma )$, and that the disjoint increasing sequences $q_1^{(i)} < \dots < q^{(i)}_{j(\iota^*_i)}$ are such that 
$$ \tau_{q_1^{(i)}}^* \star \dots \star \tau_{q_{j(\iota_i^*)}^{(i)}}^* =  \rho_{p_1^{(i)}}^* \star \dots \star \rho_{p_{j(\iota_i^*)}^{(i)}}^*  $$
is precisely the stable factorization of $\iota^*_i$.
Thus, we need only to establish that $\beta^* \leq_{fac} \gamma^*$.

We claim that $(\tau_1^*, \dots , \tau_j^* ) \geq (\rho_1^* , \dots , \rho_j^*)$.
Indeed, $(\tau_1^*, \dots , \tau_j^* ) $ is obtained from $(\rho_1^* , \dots , \rho_j^*)$ by the stabilization procedure in its subwords.
From the stabilization procedure we obtain bigger words in the lexicographical order, according to \cref{rem:smallestfact}.
This still holds true even if only applied to a subword, thus the resulting word $(\tau_1^*, \dots , \tau_j^* ) $ is lexicographically bigger than $(\rho_1^* , \dots , \rho_j^*)$.

%Indeed, assume that there exists $s$, the smallest index where $\tau_s^* \neq \rho_s^*$, and we will show that $ \tau^*_s \geq_{per} \rho^*_s$.
%Let $s = q_u^{(i)}$ for some integers $i, u$.
%We have from \cref{rem:smallestfact} that 
%$$(\rho_{p_1^{(i)}}^* , \dots , \rho_{p_{j(\iota_i^*)}^{(i)}}^*) \geq (\rho_{q_1^{(i)}}^* , \dots , \rho_{q_{j(\iota_i^*)}^{(i)}}^*) \, .$$
%However, we can see that the first $u-1$ letters of these two words coincide.
%Indeed, if $t = q^{(i)}_v$ for some $v<u$ then because $q^{(i)}_1 < \dots  < q^{(i)}_{j(\iota_i^*) }$, we have that $t<s $ and so $\tau_t^* = \rho_t^*$.
%That is $\rho_{p^{(i)}_v}^* = \rho_{q^{(i)}_v}^*$, as envisaged.
%Thus we have that $ \rho^*_{p_u^{(i)}} \geq_{per}\rho^*_{q_u^{(i)}}$, that is $ \tau^*_s \geq_{per} \rho^*_s$, and so we conclude that $(\tau_1^*, \dots , \tau_j^* ) \geq (\rho_1^* , \dots , \rho_j^*)$.

Now, if $\gamma^* = \epsilon_1^* \star \dots \star \epsilon_j^* $ is the stable factorization of $\gamma^*$, because $\gamma^* = \tau_1^* \star \dots \star \tau_j^* $ we have that from \cref{rem:smallestfact} that 
$$(\tau_1^*, \dots , \tau_j^* ) \leq (\epsilon_1^* , \dots , \epsilon_j^*) \, .$$
Thus 
$$(\rho_1^* , \dots , \rho_j^*) \leq (\epsilon_1^* , \dots , \epsilon_j^*) \, , $$
 and so $\beta^* \leq_{fac} \gamma^*$, as desired.
\end{proof}

In the following proof, we will start by showing that the given factorization of $\gamma^* $ can be assumed to be the stable factorization.
Then, we use \cref{thm:Lyndonfact} and the fact that the factorization of $\gamma^* $ is a shuffle of Lyndon words to establish the desired inequality.

\begin{proof}[Proof of \cref{lm:bplol}]
Write $j=j(\gamma^*)$.
We first see that if $\gamma^* $ has some factorization that is a word shuffle of stable words $l_1, 
\dots , l_k$, then the stable factorization is also a word shuffle of these words.
%Thus, because the stable factorization is the biggest factorization of $\gamma^*$ (see \cref{rem:smallestfact}), we can assume that $\tau_1^* \star \dots \star \tau_j^* $ is stable.

Indeed, take a factorization  $\tau_1^* \star \dots \star \tau_j^* $ such that $(\tau_1^* , \dots , \tau_j^*)$ is a word quasi-shuffle of $l_1, \dots , l_k$, and say that there is some $u \in \{1, \dots, j-1 \}$ such that this factorization is not $u$-$\oplus$-stable or is not $u$-$\ominus$-stable.
According to \cref{rem:stabred}, to show that the stable factorization is also a word shuffle of $l_1, \dots , l_k$, it suffices to show that the word resulting from swapping $\tau_u^*$ and $\tau_{u+1}^*$ in $(\tau_1^* , \dots , \tau_j^*)$ is still a word quasi-shuffle of $l_1, \dots , l_k$.

When we apply a $u$-stability reduction (see \cref{rem:stabred}), we can still find a suitable partition of $[j ]$ into $k$ many disjoint increasing sequences.
Say that $[j ]$ is partitioned into the blocks $\{q^{(i)}_1 < \dots < q^{(i)}_{j(\iota^*_i)} \}$, then there are integers $i_1, i_2, v_1, v_2$ such that $u= q^{(i_1)}_{v_1}$ and $u+1=q^{(i_2)}_{v_2}$.
Because $(\tau^*_{q^{(i)}_1}, \dots , \tau^*_{q^{(i)}_{j(\iota^*_i)}})$ is stable for each $i$, we cannot have $i_1 = i_2$.
Therefore, by swapping the elements $u, u+1$ we obtain a new partition for the new factorization, thus showing that it is a word quasi-shuffle of $l_1, \dots , l_k$.

Now let $ \rho_1^* \star \dots \star \rho_j^* $ be the stable factorization of $\gamma^* $.
Since $(\rho_1^* , \dots , \rho_j^*)$ is a word shuffle of $ l_1 , \dots , l_k$, which are Lyndon words in $\mathcal{W}(\Omega)$, we have from \cref{thm:lyndonshuf} that $(\rho_1^* , \dots , \rho_j^*) \leq l_1 \cdots l_k$, and so $\gamma^* \leq_{fac} l_1^* \star  \dots \star l_k^*$.
\end{proof}

%Let $\omega = (r_1^*, \dots , r_(\gamma^* )^* )$ be the stable factorization of $\gamma^*$, and denote $\omega_i $ for the stable factorizations of $l_i^* $.
%
%Note that $\omega $ is a shuffle of $\omega_1, \dots , \omega_k $.
%Note further that each of the words $\omega_i $ is, by definition, a Lyndon word.
%Then, by \cite{} we have that $\omega \leq \omega_1 \cdots \omega_k $ in the lexicographic order, where we compare the letters with $\leq_{per}$.
%\todo[inline]{This can directly be cited from somewhere else, but where?}
%
%Finally, note that $\omega_1 \cdots \omega_k $ is the stable factorization of $\alpha^* $, so we obtain that $\gamma^* \leq_{fac} \alpha^* $, and because $|\gamma^*| = |\alpha^*| $ and $j(\gamma^*) = j(\alpha^*) $, we have that $\gamma^* \geq_{prod} \alpha^* $, as desired.
%
%

\subsection{Primitive elements, growth rates and asymptotic analysis\label{sec:PrimEl}}

Recall that in \cref{sec:primelem} we define the space of primitive elements $P(H)$ of a Hopf algebra $H$ is the subspace of $H$ given by $\{a \in H | \Delta a = a \otimes 1 + 1 \otimes a\}$.
In the particular case of the pattern Hopf algebra $\mathcal{A}(\mathtt{MPer})$, its primitive space is spanned by $\left\{ \pat_{\pi^* } | \pi^* \text{ is  irreducible marked permutation} \right\}$, according to \cref{prop:primit}.
So, we are interested in enumerating the irreducible marked permutations.
We consider some generating functions:
\begin{itemize}
\item The power series $P^*(x) = \sum_{\pi^* \text{ marked permutation} } x^{|\pi^*|}= \sum_{n\geq 1} n \cdot n!\,  x^{n-1} $ counts marked permutations.

\item The power series $P(x) = \sum_{\pi \text{ permutation}} x^{|\pi|}= \sum_{n\geq 0} n! \, x^n $ counts permutations.

\item The power series $S^*(x) = \sum_{k\geq 0 } s_k x^k = \sum_{\pi^* \text{ irreducible marked permutation}} x^{|\pi^*|}$ counts irreducible marked permutations.
This is the main generating function that we aim to enumerate here.

\item The power series $S_o^*(x) = \sum_{k\geq 0 } so_k x^k = \sum_{\substack{\pi^* \text{ irreducible and indecomposable}\\ \text{ marked permutation }}} x^{|\pi^*|} $ counts irreducible marked permutations that are indecomposable, that is neither $\oplus$-decomposable nor $\ominus$-decomposable.

\item The power series $P^{\oplus}(x) = \sum_{\pi \text{ is } \oplus - \text{ decomposable }} x^{|\pi|}$ counts permutations that are $\oplus $-indecomposable. This also counts the permutations that are $\ominus $-indecomposable.
\end{itemize}

Because we have a unique factorization theorem on permutations under the $\oplus$ product, the coefficients of $P^{\oplus}(x)$ can be easily extracted via the following power series relation
$$\frac{1}{1 - P^{\oplus}(x)} = P(x) = \sum_{n\geq 0 } n! x^n\, , \text{ which implies } P^{\oplus} (x) = 1 -P(x)^{-1}\, . $$
%Furthermore, observe that $\frac{\partial P(x)}{\partial x} =P^*(x)$.

%We saw that if $\pi $ is an $\oplus$-indecomposable permutation, then $\bar{1}\oplus \pi $ and $\pi \oplus \bar{1} $ are simple marked permutations.
%Similarly, if $\tau $ is a $\ominus$-indecomposable permutation, then $\bar{1}\ominus \tau $ and $\tau \ominus \bar{1} $ are simple marked permutations.

\begin{obs}
Any $\oplus $-decomposable irreducible marked permutation is either of the form $\bar{1}\oplus \pi $ or of the form $\pi \oplus \bar{1} $ for $\pi$ an $\oplus$-indecomposable permutation, and symmetrically for $\ominus $-decomposable irreducible marked permutations.

Thus, we have
\begin{equation}\label{eq:pssimple}
S^*(x) - S_o^*(x) = 4P^{\oplus} (x)\, . 
\end{equation}
\end{obs}

The following proposition allows us to enumerate easily the irreducible marked permutations and the irreducible indecomposable marked permutations, as done in \cref{tbl:sseq}.

\begin{prop}[Power series of irreducible marked permutations]\label{prop:powseriesform}
$$S^*(x) = 3 + \frac{2}{P(x)^2} - \frac{1}{P'(x)} - \frac{4}{P(x)}\, .  $$
$$S_o^*(x)= -1 + \frac{2}{P(x)^2} - \frac{1}{P'(x)} \, . $$
\end{prop}

We compare this result with the enumeration of simple permutations in \cite[Equation 1]{albert03}, where the power series enumerating simple permutations is given as the solution of a functional equation that is not rational.
Thus, we expect it to be simpler to compute the coefficients explicitly.

\begin{table}[h]
\centering
\begin{tabular}{ l | c | c | c | c | c | c | c | c | c | c   }
 $n$ & 0 & 1 & 2 & 3 & 4 & 5 & 6 & 7 & 8 & 9 \\ \hline
 $so_n$ & 0& 0& 0& 8& 78& 756& 7782& 85904& 1016626& 12865852 \\ \hline
 $s_n$ & 0& 4& 4& 20& 130& 1040& 9626& 99692& 1132998& 13959224
\end{tabular}
\caption{First elements of the sequences $so_n$ and $s_n $.\label{tbl:sseq}}
\end{table}

\begin{lm}
Let $\pi^* $ be a marked permutation.
Then, there are four cases
\begin{itemize}
\item There are unique marked permutations $\sigma^*$ and $\alpha^* $ such that $\sigma^*$ is indecomposable and irreducible, and $\pi^* = \sigma^* * \alpha^* $.

\item The marked permutation $\pi^*$ is  $\oplus $-decomposable.

\item The marked permutation $\pi^*$ is  $\ominus $-decomposable.

\item $\pi^* = \bar{1}$.
\end{itemize}

In particular, we have the following equation
\begin{equation}\label{eq:simplemppow}
P^*(x) =S_o^*(x) P^*(x) + 2  ( P(x) - P^{\oplus } (x) )'  +1 \, . \end{equation}
\end{lm}

\begin{proof}
If $\pi^* $ is indecomposable, then, either $\pi^* = \bar{1}$, or there are unique marked permutations $\sigma^*$ and $\alpha^* $ such that $\sigma^*$ is indecomposable and irreducible, and $\pi^* = \sigma^* * \alpha^* $, according to \cref{lm:simplecont}.
This concludes the first part of the lemma.
%, the first factor is unique, and so we fall precisely in the first case, as desired.

Further, we observe that
$$(  P(x) - P^{\oplus } (x) )' $$
counts the marked permutations that are $\oplus $-decomposable (and by symmetry, the ones that are $\ominus$-decomposable).
\end{proof}

%
%\begin{rem}[Other ways of obtaining a power series relation]
%We are counting words on five alphabets, according to \cref{cor:simpleUFT} as follows:
%Let $\{\mathbf{a}_1, \dots \}$ be the letters representing simple marked permutations, and $\{\mathbf{b}^{(1)}_1, \dots \}$, $\{\mathbf{b}^{(2)}_1, \dots \}$, $\{\mathbf{b}^{(3)}_1, \dots \}$, $\{\mathbf{b}^{(4)}_1, \dots \}$ be respectively the letters representing simple marked permutations of the form $\bar{1}\oplus \pi$, $\pi \oplus \bar{1}$, $\bar{1}\ominus \pi$ and $\pi \ominus \bar{1}$.
%Each letter has a weight corresponding to the respective size of the simple marked permutation.
%
%Words on this alphabet are naturally enumerated by the formal power series
%$$\frac{1}{1-S^*(x)} \, . $$
%
%Whereas words that follow the following rules
%\begin{itemize}
%\item After a $\mathbf{b}^{(1)}_i$ does not come any $\mathbf{b}^{(2)}_j$,
%
%\item After a $\mathbf{b}^{(3)}_i$ does not come any $\mathbf{b}^{(4)}_j$,
%\end{itemize}
%Are enumerated by the formal power series $P^*(x)$, according to \cref{cor:simpleUFT}.
%Question: Does this relate these two power series in some way?
%\todo{Is there a way to compute these words immediately?}
%\end{rem}
%
%\begin{rem}
%Unlike the formula for the power series of simple permutations obtained in \cite{albert03}, this formula is very explicit and can be computed arithmetically.
%\end{rem}

\begin{proof}[Proof of \cref{prop:powseriesform}]
From \eqref{eq:simplemppow} along with the fact that $P^*(x) = P'(x)$, we have that
$$S_o^*(x)= -1 + 2P(x)^{-2} - P'(x)^{-1} \, , $$
$$S^*(x) = 3 + 2 P(x)^{-2} - P'(x)^{-1} - 4P(x)^{-1}\, ,  $$
as desired.
\end{proof}

\section*{Acknowledgments}

The author is  thankful for the support of the SNF grant number 200020-172515.
The author is also grateful for conversations with Yannic Vargas about pattern Hopf algebras and his insightful comments, and to the anonymous referees.
He is also thankful for conversations with Valentin F\'eray, and the great guidance throughout the process of development of this paper.

\bibliographystyle{alpha}

\end{document}